\documentclass[10pt,a4paper,leqno]{amsart}
\usepackage{xspace}
\usepackage{amsfonts,amssymb}
\usepackage{amsthm}
\usepackage[all]{xy}
\usepackage{textcomp}
\usepackage{mathrsfs}
\usepackage{stmaryrd}

\usepackage{oldgerm}
\usepackage[T1]{fontenc}
\usepackage[latin1]{inputenc}

\newcommand{\be}{\begin{otherlanguage}{english}}
\newcommand{\ee}{\end{otherlanguage}}

\theoremstyle{definition}
\newtheorem{defn}[subsection]{Definition}
\newtheorem{rem}[subsection]{Remark}
\theoremstyle{plain}
\newtheorem{lemma}[subsection]{Lemma}

\newtheorem{prop}[subsection]{Proposition}
\newtheorem{thm}[subsection]{Theorem}
\newtheorem{cor}[subsection]{Corollary}

\theoremstyle{remark}

\numberwithin{equation}{subsection}

\newcommand{\beq}{\begin{equation}}
\newcommand{\eeq}{\end{equation}}

\newcommand{\ra}{\rightarrow}
\newcommand{\lra}{\longrightarrow}

\newcommand{\xra}{\xrightarrow}
\newcommand{\hra}{\hookrightarrow}

\newcommand{\Z}{\mathbb{Z}}
\newcommand{\F}{\mathbb{F}}
\newcommand{\G}{\mathbb{G}}
\newcommand{\Q}{\mathbb{Q}}
\newcommand{\h}{\mathfrak{h}}

\newcommand{\M}{\mathbf{M}}
\newcommand{\bG}{\mathbf{G}}

\newcommand{\bS}{\mathbf{S}}
\newcommand{\bU}{\mathbf{U}}

\newcommand{\fG}{\mathfrak{G}}

\newcommand{\rM}{\mathrm {M}}
\newcommand{\fp}{\mathfrak{p}}
\newcommand{\rT}{\mathrm {T}}

\newcommand{\Comp}{\mathbf{Comp}_{\Kb_0}}
\newcommand{\Sol}{\mathrm{Sol}}

\newcommand{\cG}{\mathscr{G}}
\newcommand{\cH}{\mathscr{H}}

\newcommand{\cI}{\mathscr {I}}
\newcommand{\cN}{\mathcal{N}}
\newcommand{\Nt}{\mathcal{NP}}

\newcommand{\cO}{\mathscr {O}}

\newcommand{\cU}{\mathscr {U}}

\newcommand{\cS}{\mathscr {S}}
\newcommand{\cT}{\mathscr {T}}
\newcommand{\tv}{\mathtt{v}}

\newcommand{\lam}{\lambda}
\newcommand{\xb}{\overline{x}}
\newcommand{\tame}{\mathrm{t}}
\newcommand{\et}{\mathrm{\acute{e}t}}
\newcommand{\mult}{\mathrm{mult}}
\newcommand{\AL}{\mathfrak{A}\mathrm{L}}
\newcommand{\GL}{\mathrm{GL}}

\newcommand{\m}{\mathfrak{m}}
\newcommand{\D}{\mathcal {D}}
\newcommand{\Ks}{K^{\mathrm{sep}}}

\newcommand{\kb}{\overline{k}}
\newcommand{\Kb}{\overline{K}}

\newcommand{\Rb}{\widetilde{R'}}
\newcommand{\Ub}{\widetilde{U'}}
\newcommand{\Sb}{\widetilde{S'}}
\newcommand{\myprojlim}{\underset{\underset{n}{\longleftarrow}}{\lim}\ }

\newcommand{\car}{characteristic\xspace}
\newcommand {\resp}{\emph{resp.\xspace}}

\newcommand{\GV}{\mathfrak{G}\mathrm{V}}
\newcommand{\pLie}{p\text{-}\mathfrak{L}\mathrm{ie}}

\newcommand{\KS}{\mathrm{KS}}
\newcommand{\HW}{\varphi}
\newcommand{\Kod}{\mathrm{Kod}}
\newcommand{\ie}{\emph{i.e.} }

\newcommand{\cb}{\overline{C}}
\newcommand{\etab}{\overline{\eta}}
\newcommand{\xib}{\overline{\xi}}
\newcommand{\chib}{\overline{\chi}}
\newcommand{\rhob}{\overline{\rho}}
\newcommand{\fppf}{\mathrm{fppf}}

\DeclareMathOperator{\ord}{\mathrm{ord}}
\DeclareMathOperator{\Gal}{\mathrm{Gal}}
\DeclareMathOperator{\Coker}{\mathrm{Coker}} 
\DeclareMathOperator{\Ker}{\mathrm{Ker}}

\DeclareMathOperator{\leng}{\mathrm{leng}}
\DeclareMathOperator{\im}{\mathrm{Im}}

\DeclareMathOperator{\Lie}{\mathrm{Lie}}

\DeclareMathOperator{\Hom}{\mathrm Hom}

\DeclareMathOperator{\Ext}{\mathrm Ext}  

\DeclareMathOperator{\Id}{\mathrm Id}

\DeclareMathOperator{\cHom}{\mathscr{H}\mathit{om}} 

\DeclareMathOperator{\Aut}{Aut}
\DeclareMathOperator{\Spec}{\mathrm {Spec}}
\DeclareMathOperator{\Spf}{\mathrm {Spf}}

\begin{document}
\title[$p$-adic monodromy of Barsotti-Tate groups]
{$p$-adic monodromy of the universal deformation of a HW-cyclic
Barsotti-Tate group}
\author{TIAN Yichao}
\address{LAGA, Institut Galil\'ee, Universit\'e Paris 13,
93430 Villetaneuse, France}
\email{tian@math.univ-paris13.fr}

\begin{abstract}Let $k$ be an algebraically closed field of characteristic $p>0$, and $G$ be a Barsotti-Tate over $k$. We denote by $\bS$ the ``algebraic'' local moduli in characteristic $p$ of $G$, by $\bG$  the universal deformation of $G$ over $\bS$, and by $\bU\subset\bS$ the ordinary locus of $\bG$. The \'etale part of $\bG$ over $\bU$ gives rise to a monodromy representation $\rho_{\bG}$ of the fundamental group of $\bU$ on the Tate module of $\bG$. Motivated by a famous theorem of Igusa, we prove in this article that $\rho_{\bG}$ is surjective if $G$ is connected and HW-cyclic. This latter condition is equivalent to saying that Oort's $a$-number of $G$ equals $1$, and it is satisfied by all connected one-dimensional Barsotti-Tate groups over $k$.
\end{abstract}

\maketitle

\section{Introduction}
\subsection{} A classical theorem of Igusa says that the monodromy representation associated with a versal family 
of ordinary elliptic curves in characteristic $p>0$ is surjective \cite{Ig,Ka}. This important result has deep consequences in the theory of $p$-adic modular forms, and inpired various generalizations. 
Faltings and Chai \cite{Ch,FC} extended it to the  universal  family over the moduli space of higher dimensional principally polarized ordinary abelian varieties in characteristic $p$, and Ekedahl \cite{Ek} generalized it to the jacobian of the universal $n$-pointed curve in characteristic $p$, equipped with a symplectic level structure. We refer to Deligne-Ribet \cite{DR} and Hida \cite{Hi} for other generalizations to some moduli spaces of PEL-type and their arithmetic applications.
Though it has been formulated in a global setting, the proof of Igusa's theorem is purely local, 
and it has got also local generalizations. 
Gross \cite{Go} generalized it to one-dimensional formal $\cO$-modules over a complete discrete valuation ring of characteristic $p$, where  $\cO$  is the integral closure of $\Z_p$ in a finite extension of $\Q_p$. We refer to Chai \cite{Ch} and Achter-Norman \cite{AN} for more results on local monodromy of Barsotti-Tate groups. 
Motivated by these results, it has been longly expected/conjectured that the monodromy of a \emph{versal}
family of ordinary Barsotti-Tate groups in characteristic $p>0$ is maximal.
The aim of this paper is to prove the surjectivity of the monodromy representation 
associated with the universal deformation in characteristic $p$ of a certain class of Barsotti-Tate groups.

\subsection{} To describe our main result, we introduce first the notion of HW-cyclic Barsotti-Tate groups. Let $k$ be an algebraically closed field of characteristic $p>0$, and $G$ be a Barsotti-Tate group over $k$. We denote by $G^\vee$ the Serre dual of $G$, and by $\Lie(G^\vee)$ its Lie algebra. 
The Frobenius homomorphism of $G$ (or dually the Verschiebung of $G^\vee$) induces a semi-linear endomorphism $\varphi_G$ on $\Lie(G^\vee)$,  
called the Hasse-Witt map of $G$ \eqref{BT-HW}. 
We say that $G$ is \emph{HW-cyclic}, if $c=\dim(G^\vee)\geq 1$ and there is a $v\in \Lie(G^\vee)$ such that $v,\varphi_G(v),\cdots, \varphi^{c-1}_G(v)$ form a basis of $\Lie(G^\vee)$ over $k$ (\ref{defn-cyclic}).
We prove in \ref{prop-HW-a} that $G$ is HW-cyclic and non-ordinary if and only if the $a$-number of $G$, 
defined previously by Oort, equals $1$. We can construct HW-cyclic Barsotti-Tate groups as follows. 
Let $r,s$ be relatively prime integers such that $0\leq s\leq r$ and $r\neq 0$, $\lambda=s/r$, $G^\lambda$ be the Barsotti-Tate group over $k$ whose (contravariant) Dieudonn\'e module is generated by an element $e$ over the non-commutative Dieudonn\'e ring with the relation $(F^{r-s}-V^s)\cdot e=0$ \eqref{HW-exem}. It is easy to see that $G^\lambda$ is HW-cyclic for any $0<\lambda<1$. 
Any connected Barsotti-Tate group over $k$ of dimension $1$ and  height $h$ is isomorphic to $G^{1/h}$ \cite[Chap.IV \S8]{De}. 

Let $G$ be a  Barsotti-Tate group of dimension $d$ and height $c+d$ over $k$; assume $c\geq 1$. We denote by  $\bS$ the ``algebraic'' local moduli of $G$ in characteristic $p$, and by  $\bG$ be the  universal deformation of $G$ over $\bS$
(cf. \ref{defn-moduli}). The scheme $\bS$ is affine of ring $R\simeq k[[(t_{i,j})_{1\leq i \leq c,1\leq j\leq d}]]$, and the Barsotti-Tate group $\bG$ is obtained by algebraizing the formal universal deformation of $G$ over $\Spf(R)$ (\ref{cor-alg-univ}). Let $\bU$ be the ordinary locus of $\bG$ (\ie the open subscheme of $\bS$ parametrizing the ordinary fibers of $\bG$), and $\etab$  a geometric point over the generic point of $\bU$. For any integer $n\geq 1$, we denote  by $\bG(n)$ the kernel of the multiplication by $p^n$ on $\bG$, and by 
\[
\rT_p(\bG,\etab)=\varprojlim_n\bG(n)(\etab)
\] the Tate module of $\bG$ at  $\etab$. This is a free $\Z_p$-module of rank $c$. We consider the monodromy representation attached to the \'etale part of $\bG$ over $\bU$
\begin{equation}\label{mono-rep-univ}
\rho_{\bG}:\pi_1(\bU,\etab)\ra \Aut_{\Z_p}(\rT_p(\bG,\etab))\simeq \GL_{c}(\Z_p).\end{equation}
The aim of this paper is to prove the following~:

\begin{thm}\label{thm-main} If $G$ is connected and HW-cyclic, then
the monodromy representation $\rho_\bG$ is surjective.
\end{thm}

 Igusa's theorem mentioned above corresponds to Theorem \ref{thm-main} for $G=G^{1/2}$ (cf. \ref{thm-Igusa}). My interest in the $p$-adic monodromy problem started with the second part of my PhD thesis \cite{tian}, 
where I guessed \ref{thm-main} for $G=G^{\lam}$ with $0< \lam <1$ and proved it for $G^{1/3}$. 
After I posted the manuscript on ArXiv \cite{tian2},
Strauch proved the one-dimensional case of \ref{thm-main} by using Drinfeld's level structures \cite[Theorem 2.1]{str}.  Later on, Lau \cite{lau} proved \ref{thm-main} without the assumption that $G$ is HW-cyclic. 
By using the  Newton stratification of the universal deformation space of $G$ due to Oort, Lau reduced
the higher dimensional case to the one-dimensional case treated by Strauch. 
In fact, Strauch and Lau considered more generally the monodromy representation 
over each $p$-rank stratum of the universal deformation space.   Recently, Chai and Oort \cite{CO} proved the maximality of the $p$-adic monodromy  over each ``central leaf'' in the moduli space of abelian varieties which is not contained in the supersingular locus.
In this paper, we provide first a different proof of the one-dimensional case of \ref{thm-main}. 
Our approach is purely  characteristic $p$, while Strauch used Drinfeld's level structure in characteristic $0$. 
Then by following Lau's strategy, we give a new (and easier) argument to reduce the general case of  \ref{thm-main}
to the one-dimensional case for HW-cyclic groups. The essential part of our argument is a versality criterion by Hasse-Witt maps of deformations of a connected one-dimensional Barsotti-Tate group (Prop. \ref{prop-HW-versal}). This criterion can be considered as a generalization of another theorem of Igusa 
which claims that the Hasse invariant of a versal family of elliptic curves in characteristic $p$ has simple zeros. Compared with Strauch's approch, our characteristic $p$ approach has the advantage of giving also results 
on the monodromy of Barsotti-Tate groups over a discrete valuation ring of characteristic $p$.

\subsection{} Let $A=k[[\pi]]$ be the ring of formal power series over $k$ in the variable $\pi$, $K$ its fraction field, and $\tv$  the valuation on $K$ normalized by $\tv(\pi)=1$. We fix an algebraic closure $\Kb$ of $K$, and let $\Ks$ be the separable closure of $K$ contained in $\Kb$, $I$ be the Galois group of $\Ks$ over $K$, $I_p\subset I$ be the wild inertia subgroup, and $I_t=I/I_p$ the tame inertia group. For every integer $n\geq 1$, there is a canonical surjective character $\theta_{p^n-1}:I_t\ra \F^{\times}_{p^n}$ \eqref{galois-char}, where $\F_{p^n}$ is the finite subfield of $k$ with $p^n$ elements.

We put $S=\Spec(A)$. Let $G$ be a Barsotti-Tate group over $S$, $G^\vee$ be its Serre dual, and  $\Lie(G^\vee)$ the Lie algebrasof $G^\vee$. Recall that the Frobenius homomorphism of $G$  induces a semi-linear endomorphism $\varphi_G$ of $\Lie(G^\vee)$, called the Hasse-Witt map of $G$. We define  $h(G)$ to be the valuation of the determinant of a matrix of $\varphi_G$, and call it the \emph{Hasse invariant} of $G$ \eqref{defn-hw-index}. We see easily that $h(G)=0$ if and only if $G$ is ordinary over $S$, and $h(G)<\infty$ if and only if $G$ is generically ordinary. If  $G$  is connected of height $2$ and dimension $1$, then $h(G)=1$ is equivalent to that $G$ is versal \eqref{thm-Igusa}.

\begin{prop} Let $S=\Spec(A)$ be as above, $G$ be a connected HW-cyclic Barsotti-Tate group with  Hasse invariant $h(G)=1$, and  $G(1)$ the kernel of the multiplication by $p$ on $G$. Then the action of $I$ on $G(1)(\Kb)$ is tame; moverover, $G(1)(\Kb)$ is an $\F_{p^c}$-vector space of dimension $1$ on which the induced action of $I_t$ is given by the surjective character $\theta_{p^c-1}:I_t\ra \F_{p^c}^\times$.
\end{prop}

This proposition is an analogue in characteristic $p$ of Serre's result \cite[Prop. 9]{Se} on the tameness of the monodromy associated with one-dimensional formal groups over a trait of mixed characteristic. We refer to \ref{prop-mono-trait} for the proof of this proposition and more results on the $p$-adic monodromy of HW-cyclic Barsotti-Tate groups over a trait in characteristic $p$. 

\subsection{} This paper is organized as follows. In Section 2, we review some well known facts on ordinary Barsotti-Tate groups. Section 3 contains some  preliminaries on the Dieudonn\'e theory and the deformation theory of Barsotti-Tate groups. In Section 4, after establishing some basic properties of HW-cyclic groups, we give the fundamental relation between the versality of a Barsotti-Tate group and the coefficients of its Hasse-Witt matrix (Prop. \ref{prop-HW-versal}). Section 5 is devoted to the study of the monodromy of a HW-cyclic Barsotti-Tate group over a  complete trait of characteristic $p$. Section 6 is totally elementary, and contains a criterion \eqref{lemma-gp-1} for the surjectivity of a homomorphism from a profinite group to $\GL_n(\Z_p)$. In Section 7, we prove the one-dimensional case of Theorem \ref{thm-main}. Finally in Section 8, we follow Lau's strategy and complete the proof of \ref{thm-main} by reducing the general case to the one-dimensional case treated in Section 7. 

\subsection{Acknowledgement} This paper is an expanded version of the second part of my Ph.D. thesis at University Paris 13. I would like to express my great gratitude to my thesis advisor Prof. A. Abbes for his encouragement during this work, and also for his various helpful comments on earlier versions of this paper. I also thank  heartily E. Lau, F. Oort and  M. Strauch for interesting discussions and valuable suggestions.

\subsection{Notations}\label{Notations} Let $S$ be a scheme of characteristic $p>0$.  A {\em BT-group} over $S$
stands for a Barsotti-Tate group over $S$. Let $G$ be a commutative
finite group scheme (\resp a BT-group) over $S$. We denote by
$G^\vee$ its Cartier dual (\resp its Serre dual), by $\omega_G$ the 
sheaf of invariant differentials of $G$ over $S$, and by $\Lie(G)$
the sheaf of Lie algebras of $G$. If $S=\Spec(A)$ is affine and  there is no risk of confusions, we also use $\omega_G$ and $\Lie(G)$ to denote the correponding $A$-modules of global sections.
We put $G^{(p)}$ the pull-back of $G$ by the absolute Frobenius
of $S$, $F_G\colon G\ra G^{(p)}$ the Frobenius homomorphism and
$V_G\colon G^{(p)}\ra G$ the Verschiebung homomorphism. If $G$ is a
BT-group and  $n$ an integer $\geq 1$, we denote by $G(n)$ the
kernel of { the} multiplication by $p^n$ on $G$; we have
$G^\vee(n)=(G^\vee)(n)$ by definition. For an $\cO_S$-module $M$, we denote by $M^{(p)}=\cO_S\otimes_{F_S}M$ the scalar extension of $M$ by the absolute Frobenius of $\cO_S$. If $\varphi: M\ra N$ be a semi-linear homomorphism of $\cO_S$-modules, we denote by $\widetilde{\varphi}:M^{(p)}\ra N$ the linearization of $\varphi$, \ie we have $\widetilde{\varphi}(\lambda\otimes x)=\lambda\cdot \varphi(x)$, where  $\lambda$ (\resp $x$) is a local section of $\cO_S$ (\resp of $M$).

Starting from Section 5,  $k$ will denote an algebraically closed field of characteristic $p>0$.

\section{Review of ordinary Barsotti-Tate groups}
In this section, $S$ denotes a scheme of \car $p>0$.
\subsection{}\label{Lie-alg}
Let $G$ be a commutative group scheme, locally free of finite type
over $S$. We have a canonical isomorphism of coherent
$\cO_S$-modules  \cite[2.1]{Il}
\begin{equation}\label{dual-Groth}\Lie(G^\vee)\simeq
\cHom_{S_{\mathrm{fppf}}}(G,\G_a) ,
\end{equation}
where $\cHom_{S_{\mathrm{fppf}}}$ is the sheaf of homomorphisms in
the category of abelian $\fppf$-sheaves over $S$, and $\G_a$ is the
additive group scheme. Since $\G_a^{(p)}\simeq \G_a$, the Frobenius homomorphism of $\G_a$ induces
an endomorphism
\begin{equation}\label{hw-finite}
\varphi_G:\Lie(G^\vee)\ra \Lie(G^\vee),
\end{equation} semi-linear with respect
to the absolute Frobenius map $F_S:\cO_S\ra \cO_S$; we call it the
\emph{Hasse-Witt} map of $G$. By the functoriality of Frobenius, $\varphi_G$ is also the canonical map induced by the Frobenius of $G$, or dually by the Verschiebung of $G^\vee$.

\subsection{}\label{sect-GV}  By a \emph{commutative $p$-Lie algebra} over $S$, we mean a pair
$(L,\varphi)$, where $L$ is an $\cO_S$-module  locally free of finite type, and $\varphi:L\ra L$ is a  semi-linear endomorphism with
respect to the absolute Frobenius $F_S:\cO_S\ra \cO_S$. When there
is no risk of confusions, we omit $\varphi$ from the notation. We
denote by $\pLie_S$ the category of commutative $p$-Lie algebras
over $S$.

Let $(L,\varphi)$ be an object of $\pLie_S$. We denote by
\[
\cU(L)=\mathrm{Sym}(L)=\oplus_{n\geq 0}\ \mathrm{Sym}^{n} (L),
\] the symmetric algebra of $L$ over $\cO_S$.   Let $\cI_p(L)$ be
the ideal sheaf of $\cU(L)$  defined, for an open subset $V\subset
S$, by
\[
\Gamma(V,\cI_p(L))=\{x^{\otimes p}-\varphi(x)\ ;\  x\in \Gamma(V,\cU(L))\},
\]
where $x^{\otimes p}=x\otimes x\otimes\cdots\otimes x\in \Gamma(V,\mathrm{Sym}^{p} (L))$.
We put $\cU_p(L)=\cU(L)/\cI_p(L)$, and call it the
\emph{$p$-enveloping algebra of $(L,\varphi)$}. We endow $\cU_p(L)$ with the
structure of a Hopf-algebra with the comultiplication given by
$\Delta(x)=1\otimes x +x\otimes 1 $ and the coinverse given by
$i(x)=-x$.

Let $G$ be a commutative group scheme, locally free  of finite type
over $S$. We say that $G$ is \emph{of coheight one} if the
Verschiebung $V_G: G^{(p)}\ra G$ is the zero homomorphism. We denote by $\GV_S$
the category of such objects. For  an object $G$ of
$\GV_S$, the Frobenius  $F_{G^\vee}$ of  $G^\vee$ is zero, so the Lie algebra $\Lie(G^\vee)$ is
locally free of finite type over $\cO_S$ (\cite{DG} $\mathrm{VII_A}$
Th\'eo. 7.4(iii)). The Hasse-Witt map of $G$ \eqref{hw-finite}
endows $\Lie(G^\vee)$ with  a commutative $p$-Lie algebra structure
over $S$.

\begin{prop}[\cite{DG} $\mathrm{VII_{A}}$, Th\'eo. 7.2 et
7.4]\label{GV-pLie} The functor $\GV_S\ra \pLie_S$  defined by
$G\mapsto \Lie(G^\vee)$ is an anti-equivalence of categories;  a
quasi-inverse is given by 
$(L,\varphi)\mapsto \Spec (\cU_p(L))$.
\end{prop}

\subsection{}\label{desc-ex}
Assume $S=\Spec(A)$ affine. Let $(L,\varphi)$ be an object of $
\pLie_S$ such that $L$ is free of rank $n$ over $\cO_S$,
$(e_1,\cdots, e_n)$ be a basis of $L$ over $\cO_S$,
$(h_{ij})_{1\leq i, j\leq n}$ be the matrix of $\varphi$ under the basis $(e_1,\cdots,e_n)$, \ie
 $\varphi(e_j)=\sum_{i=1}^n
h_{ij}e_i$ for $1\leq j\leq n$. Then the group scheme associated to $(L,\varphi)$ is
explicitly given by
\[
\Spec(\cU_p(L))=\Spec  \biggl (A[X_1,\cdots,X_n]/(X_j^p-\sum_{i=1}^nh_{ij}X_i)_{1\leq j\leq n}\biggr),
\]
with the comultiplication $\Delta(X_j)=1\otimes X_j+X_j\otimes 1$.
By the Jacobian criterion of \'etaleness [EGA $\mathrm{IV_0}$
22.6.7], the finite group scheme $\Spec(\cU_p(L))$ is \'etale over
$S$ if and only if the matrix $(h_{ij})_{1\leq i, j\leq n}$ is
invertible. This condition is equivalent to that the  linearization
of $\varphi$  is  an isomorphism.

\begin{cor}\label{cor-etale-GV} An object $G$ of
$\GV_S$ is \'etale over $S$, if and only if the linearization  of
its Hasse-Witt map \eqref{hw-finite} is an isomorphism.
\end{cor}

\begin{proof}
The problem being local over  $S$, we may assume $S$  affine and
$L=\Lie(G^\vee)$ free over $\cO_S$. By Theorem \ref{GV-pLie}, $G$ is
isomorphic to $\Spec(\cU_p(L))$,  and we conclude by the last remark
of \ref{desc-ex}.
\end{proof}

\subsection{}  Let $G$ be a BT-group over $S$ of height $c+d$ and dimension $d$, $G^\vee$ be its Serre dual. The Lie algebra $\Lie(G^\vee)$ is an $\cO_S$-module locally free of rank $c$, and canonically identified with $\Lie(G^\vee(1))$(\cite{BBM} 3.3.2).
We define the {\em Hasse-Witt map} of $G$
\begin{equation}\label{BT-HW}
\HW_G:\Lie(G^\vee)\ra\Lie(G^\vee)
\end{equation}
 to be that of $G(1)$ \eqref{hw-finite}.

\subsection{}Let $k$ be a field of characteristic $p>0$, $G$ be a BT-group over $k$. Recall that we have a canonical exact
sequence of BT-groups over $k$
\begin{equation}\label{decomp-BT}0\ra G^{\circ}\ra G\ra G^{\et}\ra 0\end{equation}
with $G^{\circ}$ connected and $G^\et$ \'etale (\cite{De} Chap.II, \S 7). This induces an exact sequence of Lie algebras 
\begin{equation}\label{decomp-Lie}
0\ra \Lie(G^{\et\vee})\ra \Lie(G^\vee)\ra \Lie(G^{\circ\vee})\ra 0, 
\end{equation}
 compatible with Hasse-Witt maps.

\begin{prop}\label{prop-etale-HW} Let $k$ be a field of characteristic $p>0$, $G$ be a BT-group over $k$. Then
$\Lie(G^{\et\vee})$  is the unique maximal $k$-subspace $V$ of   $ \Lie(G^\vee)$ with the following properties:

\emph{(a)} $V$ is stable under $\varphi_G$;

\emph{(b)} the restriction of $\varphi_G$ to $V$ is injective. 
\end{prop}

\begin{proof} It is clear that  $\Lie(G^{\et\vee})$ satisfies property (a). We note that the Verschiebung of  $G^\et(1)$ vanishes; so $G^\et(1)$ is in the category $\GV_{\Spec(k)}$. Since $k$ is a field, \ref{cor-etale-GV} implies that  the restriction of $\varphi_G$ to $\Lie(G^{\et\vee})$, which coincides with $\varphi_{G^\et}$, is injective. This proves that $\Lie(G^{\et\vee})$ verifies (b). Conversely, let $V$ be an arbitrary $k$-subspace of $\Lie(G^\vee)$ with properties (a) and (b). We have to show that $V\subset \Lie(G^{\et\vee})$. Let $\sigma$ be the Frobenius endomorphism of $k$. If $M$ is a $k$-vector space, for each integer $n\geq 1$, we put $M^{(p^n)}=k\otimes_{\sigma^n}M$, \ie we have $1\otimes ax=\sigma^n(a)\otimes x$ in $k\otimes_{\sigma^n}M$. Since  $\varphi_G|_V:V\ra V$ is injective by assumption,  the linearization  $\widetilde{\varphi^n_{G}}|_{V^{(p^n)}}:V^{(p^n)}\ra V$ of $\varphi^n_G|_V$ is injective (hence bijective) for any $n\geq 1$. We have $V=\widetilde{\varphi_G^n}(V^{(p^n)})$.  Since $G^\circ$ is connected, there is an integer $n\geq 1$ such that the $n$-th iterated Frobenius $F^n_{G^\circ(1)}:G^\circ(1)\ra G^\circ(1)^{(p^n)}$ vanishes. Hence by definition, the linearized $n$-iterated Hasse-Witt map $\widetilde{\varphi^n_{G^\circ}}:\Lie(G^{\circ\vee})^{(p^n)}\ra \Lie(G^{\circ\vee})$ is zero. By the compatibility of Hasse-Witt maps,  we have $\widetilde{\varphi_G^n}(\Lie(G^\vee)^{(p^n)})\subset \Lie(G^{\et\vee})$; in particular, we have $V=\widetilde{\varphi^n_G}(V^{(p^n)})\subset \Lie(G^{\et\vee})$. This completes the proof.  
\end{proof}

\begin{cor}\label{cor-nilp-HW} Let $k$ be a field of characteristic $p>0$, $G$ be a BT-group over $k$. Then $G$ is connected if and only if $\varphi_G$ is nilpotent.
\end{cor}
\begin{proof}
In the proof of the proposition, we have seen that the Hasse-Witt map of the connected part of $G$ is nilpotent. So the ``only if'' part is verified. Conversely, if $\varphi_G$ is nilpotent,  $\Lie(G^{\et\vee})$ is zero by the proposition. Therefore $G$ is connected.
\end{proof}

\begin{defn} Let $S$ be a scheme of characteristic $p>0$,  $G$ be a BT-group over $S$. We say that $G$ is \emph{ordinary}
if there exists an exact sequence of BT-groups over $S$
\begin{equation}\label{decom-ord}
0\ra G^{\mult}\ra G\ra G^{\et}\ra 0,
\end{equation}
such that $G^{\mult}$ is  multiplicative  and $G^{\et}$ is \'etale.
\end{defn}

 We note that when it exists, the exact sequence
\eqref{decom-ord}
 is unique up to a unique isomorphism, because there is no
non-trivial homomorphisms between a multiplicative BT-group and an
\'etale one in characteristic $p>0$. The property of being ordinary is clearly stable under arbitrary
base change and Serre duality. If $S$ is the spectrum of a field of characteristic $p>0$,   $G$ is
ordinary if and only if its connected part $G^{\circ}$ is of
multiplicative type.

\begin{prop}\label{prop-ord} Let $G$ be a BT-group over $S$. The following
conditions are equivalent:

\emph{(a)} $G$ is ordinary over $S$.

\emph{(b)} For every $x\in S$, the fiber $G_x=G\otimes_S \kappa(x)$ is ordinary over $\kappa(x)$.

\emph{(c)} The finite group scheme $\Ker V_G$ is \'etale  over $S$.

\emph{(c')} The finite group  scheme $\Ker F_G$ is of multiplicative type
over $S$.

\emph{(d)} The linearization  of the Hasse-Witt map $\HW_G$ 
is an isomorphism.
\end{prop}

 First, we prove the following lemmas.

\begin{lemma}\label{lemma-et-mul} Let $T$ be a scheme,  $H$ be a commutative group
scheme locally free  of finite type over $T$. Then $H$ is  \'etale (\resp  of
multiplicative type)  over $T$ if and only if, for  every
$x\in T$, the fiber $H\otimes_T\kappa(x)$ is \'etale (\resp  of multiplicative type) over
$\kappa(x)$.
\end{lemma}
\begin{proof}We will  consider only the \'etale case; the multiplicative
case follows by duality. Since $H$ is $T$-flat, it is \'etale over
$T$ if and only if it is unramified over $T$. By [EGA IV 17.4.2],
this condition is equivalent to that  $H\otimes_T\kappa(x)$ is unramified over $\kappa(x)$ for every point $x\in T$. Hence  the conclusion follows.
\end{proof}

\begin{lemma}\label{lemma-KerV} Let $G$ be a BT-group over $S$. Then
$\Ker V_G$ is an object of the category $\GV_S$, \ie it is locally
free of finite type over $S$, and its Verschiebung  is  zero. Moreover, we have a canonical isomorphism $(\Ker
V_G)^\vee\simeq \Ker F_{G^\vee}$, which induces an isomorphism of
Lie algebras $\Lie\bigl((\Ker V_G)^\vee\bigr)\simeq \Lie (\Ker
F_{G^\vee})= \Lie (G^\vee)$, and the Hasse-Witt map
\eqref{hw-finite} of $\Ker V_G$ is identified with $\HW_G$
\eqref{BT-HW}.
\end{lemma}
\begin{proof} The group scheme $\Ker V_G$ is locally free of finite
type over $S$ (\cite{Il} 1.3(b)), and we have a commutative diagram
\[\xymatrix{(\Ker V_G)^{(p)}\ar[rr]^{V_{\Ker V_G}}\ar@{^(->}[d] &&\Ker V_G\ar@{^(->}[d]\\
(G^{(p)})^{(p)}\ar[rr]^{V_{G^{(p)}}}&&G^{(p)}}\] By the
functoriality of Verschiebung, we have $ V_{G^{(p)}}=( V_{G})^{(p)}$
and $\Ker V_{G^{(p)}}=(\Ker V_G)^{(p)}$. Hence the composition of
the left vertical arrow with $V_{G^{(p)}}$ vanishes, and  the
Verschiebung of $\Ker V_{G}$ is zero.

By  Cartier duality, we have  $(\Ker V_G)^\vee=\Coker
(F_{G^\vee(1)})$. Moreover,  the exact sequence
\[\cdots \ra G^\vee(1)\xra{F_{G^\vee(1)}}\bigl(G^\vee(1)\bigr)^{(p)}\xra{V_{G^\vee(1)}}G^\vee(1)\ra \cdots,\]
  induces a canonical isomorphism
 \begin{equation}\label{isom-Ker-Coker}\Coker
(F_{G^\vee(1)})\xra{\sim} \im (V_{G^\vee(1)})=\Ker
F_{G^\vee(1)}=\Ker F_{G^\vee}.\end{equation} Hence, we deduce that
 \begin{equation}\label{isom-Ker-Coker2}(\Ker
V_G)^\vee\simeq\Coker(F_{G^\vee(1)})\xra{\sim} \Ker
F_{G^\vee}\hookrightarrow G^\vee(1).\end{equation} Since the natural
injection $\Ker F_{G^\vee}\ra G^\vee(1)$ induces an isomorphism of
Lie algebras, we get \begin{equation}\label{isom-Lie}\Lie\bigl((\Ker
V_G)^\vee\bigr) \simeq \Lie (\Ker F_{G^\vee})=\Lie
(G^\vee(1))=\Lie(G^\vee).\end{equation} It remains to prove the
compatibility of the Hasse-Witt maps with \eqref{isom-Lie}. We note
that the dual of the morphism \eqref{isom-Ker-Coker2}
 is the canonical map $F:G(1)\ra \Ker V_G=\im
(F_{G(1)})$ induced by $F_{G(1)}$. Hence by \eqref{dual-Groth}, the
isomorphism \eqref{isom-Lie} is identified with the functorial map
\[\cHom_{S_{\fppf}}(\Ker V_G,\G_a)\ra \cHom_{S_{\fppf}}(G(1),\G_a)\]
induced by $F$, and its compatibility with the Hasse-Witt maps
follows easily from the definition \eqref{hw-finite}.
\end{proof}

\begin{proof}[Proof of \ref{prop-ord}]  (a)$\Rightarrow$(b). Indeed,
the ordinarity of $G$ is stable by base change.

  (b)$\Rightarrow$(c). By Lemma \ref{lemma-et-mul},  it suffices to verify that for every point $x\in S$,
 the fiber $(\Ker V_G)\otimes_S{\kappa(x)}\simeq \Ker V_{G_x}$ is \'etale over
 $\kappa(x)$. Since $G_x$ is assumed to be  ordinary,  its connected part
 $(G_x)^{\circ}$ is multiplicative.  Hence, the Verschiebung of
 $(G_x)^{\circ}$ is an isomorphism, and  $\Ker V_{G_x}$ is canonically isomorphic to
 $ \Ker V_{G_x^\et}\subset (G_x^\et)^{(p)}\simeq (G_x^{(p)})^\et$, so  our assertion follows.

 $(c)\Leftrightarrow(d)$. It follows
immediately from Lemma \ref{lemma-KerV} and Corollary
\ref{cor-etale-GV}.

 (c)$\Leftrightarrow$(c').
  By \ref{lemma-et-mul}, we may assume that $S$ is  the spectrum of a field.
 So the category of commutative
 finite group schemes over $S$ is abelian. We will just prove
 (c)$\Rightarrow$(c'); the converse can be proved by duality.
 We have  a fundamental  short exact
 sequence of finite group schemes
 \begin{equation}\label{exseq-fund-BT}0\ra \Ker F_G\ra G(1)\xra{F} \Ker
 V_G\ra0,
 \end{equation} where $F$ is induced by $F_{G(1)}$,
 That induces a commutative diagram
 \[\xymatrix{0\ar[r]&\bigl(\Ker F_G\bigr)^{(p)}\ar[d]^{V'}\ar[r]&\bigl(G(1)\bigr)^{(p)}\ar[r]^{F^{(p)}}\ar[d]^{V_{G(1)}}
 &\bigl(\Ker V_G\bigr)^{(p)}\ar[r]\ar[d]^{V''}&0\\
 0\ar[r]&\Ker F_G\ar[r]& G(1)\ar[r]^{F}&\Ker V_G\ar[r]&0}\]
 where vertical arrows are the Verschiebung homomorphisms. We have seen
  that  $V''=0$ (\ref{lemma-KerV}).
 Therefore, by the snake lemma, we have a long exact sequence
 \begin{equation}\label{long-ext-seq}0\ra \Ker V'\ra \Ker V_{G(1)}\xra{\alpha}
 \bigl(\Ker V_G\bigr)^{(p)}\ra \Coker V'\ra \Coker V_{G(1)}\xra{\beta}\Ker V_G\ra
 0,
 \end{equation}
where the map $\alpha$  is the Frobenius of $\Ker V_G$ and $\beta$
is the composed isomorphism
$$\Coker (V_{G(1)})\simeq G(1)/\Ker F_{G(1)}\xra{\sim} \im (F_{G(1)})\simeq
\Ker V_G.$$ Then condition (c) is equivalent to that $\alpha$ is an
isomorphism; it implies that $\Ker V'=\Coker V'=0$, \ie the
Verschiebung of $\Ker F_G$ is an isomorphism, and hence  (c').

 (c)$\Rightarrow$(a). For every integer $n>0$, we denote
by $F^n_G$ the composed homomorphism
\[G\xra{F_G}G^{(p)}\xra{F_{G^{(p)}}}\cdots\xra{F_{G^{(p^{n-1})}}}G^{(p^n)},\] and by $V^n_G$
the composed homomorphism
\[G^{(p^n)}\xra{V_{G^{(p^{n-1})}}}G^{(p^{n-1})}\xra{V_{G^{(p^{n-2})}}}\cdots\xra{V_{G}}G;\]
 $F_G^n$ and $V_G^n$ are isogenies of BT-groups. From the relation
$V^n_G\circ F^n_G=p^n$, we deduce an exact sequence
\beq\label{exseq-FV}0\ra \Ker F^n_G\ra G(n)\xra{F^n} \Ker V^n_G\ra
0,\eeq where $F^n$ is induced by $F_{G}^n$.  For $1\leq j<n$, we
have a commutative diagram
\beq\label{diag-V_G} \xymatrix{G^{(p^n)}\ar[rr]^{V^{n-j}_{G^{(p^j)}}}\ar[rd]_{V_G^n}&&G^{(p^j)}\ar[ld]^{V_{G}^j}\\
&G.}\eeq One notices by the functoriality of  Verschiebung that
$\Ker V^{n-j}_{G^{(p^j)}}=(\Ker V^{n-j}_G) ^{(p^j)}$. Since all maps
in \eqref{diag-V_G} are isogenies, we have an exact sequence
\beq\label{exseq-V}0\ra (\Ker V^{n-j}_G)^{(p^j)}\xra{i'_{n-j,n}}
\Ker V^n_G\xra{p_{n,j}} \Ker V^j_G\ra 0. \eeq  Therefore, condition
(c) implies by induction that $\Ker V^n_G$ is an \'etale group
scheme over $S$. Hence the $j$-th iteration of the Frobenius $\Ker
V^{n-j}_G\ra (\Ker V^{n-j}_G)^{(p^j)}$ is  an isomorphism,
and  $\Ker V^{n-j}_G$ is identified with a closed subgroup scheme of $\Ker
V^{n}_G $ by the composed map
\[
i_{n-j,n}:\;\Ker V^{n-j}_G\xra{\sim} (\Ker
V^{n-j}_{G})^{(p^j)}\xra{i'_{n-j,n}}\Ker V^n_G.
\]
We claim that the kernel of the multiplication by $p^{n-j}$ on $\Ker
V^n_G$ is $\Ker V^{n-j}_G$. Indeed, from the relation $p^{n-j}\cdot
\Id_{G^{(p^n)}}=F^{n-j}_{G^{(p^j)}}\circ V^{n-j}_{G^{(p^j)}}$, we
deduce a commutative diagram (without dotted arrows)
\begin{equation}\label{diag-1}
\xymatrix{\Ker
V^n_G\ar[rr]\ar[dd]_-(0.7){p^{n-j}}\ar@{-->}[rd]^{p_{n,j}}&&G^{(p^n)}
\ar[dd]_-(0.7){p^{n-j}}\ar[rd]^{V^{n-j}_{G^{(p^j)}}}\\
&\Ker
V^j_{G}\ar@{-->}[rr]\ar@{-->}[ld]^{i_{j,n}}&&G^{(p^j)}\ar[ld]^{F^{n-j}_{G^{(p^{j})}}}\\
\Ker V_G^n\ar[rr]&&G^{(p^n)}.}
\end{equation}
It follows from \eqref{exseq-V} that the subgroup  $\Ker V^n_G$ of $
G^{(p^n)}$ is sent by $V^{n-j}_{G^{(p^j)}}$ onto $\Ker V^j_G$.
Therefore diagram \eqref{diag-1} remains commutative when completed
by the dotted arrows, hence our claim.
 It follows from the claim that $(\Ker V_G^n)_{n\geq 1}$ constitutes
  an \'etale BT-group over
$S$, denoted by $G^\et$. By duality, we have an exact sequence
\begin{equation}\label{exseq-F}0\ra \Ker
F^j_G\ra \Ker F^n_G\ra (\Ker F^{n-j}_G)^{(p^j)}\ra 0.\end{equation}
Condition (c') implies by induction that $\Ker F^n_G$ is of
multiplicative type. Hence the $j$-th iteration of Verschiebung
$(\Ker F^{n-j}_G)^{(p^j)}\ra \Ker F^{n-j}_G$ is an isomorphism. We
deduce  from \eqref{exseq-F} that $(\Ker F^n_G)_{n \geq 1}$ form a
multiplicative BT-group over $S$ that we denote  by $G^{\mult}$.
Then the exact sequences \eqref{exseq-FV} give a decomposition of
$G$ of the form \eqref{decom-ord}.
\end{proof}

\begin{cor}\label{cor-ord-loc} Let $G$ be a BT-group over $S$, and $S^{\ord}$ be the
locus in $S$ of the points $x\in S$ such that
$G_x=G\otimes_S\kappa(x)$ is ordinary over $\kappa(x)$. Then
$S^{\ord}$ is open in $S$, and the canonical inclusion $S^{\ord}\ra
S$ is affine.
\end{cor}
The open subscheme $S^{\ord}$ of $S$ is called the \emph{ordinary
locus }of $G$.

\section{Preliminaries on Dieudonn\'e Theory and Deformation Theory}

\subsection{}\label{pre-Dieud} We will use freely the conventions of \ref{Notations}. Let $S$ be a scheme of characteristic $p>0$, $G$ be  a Barsotti-Tate group  over $S$, and $\M(G)$ be the coherent $\cO_S$-module obtained by evaluating the (contravariant) Dieudonn\'e crystal of $G$ at the trivial divided power immersion $S\hra S$. Recall that  $\M(G)$ is an $\cO_S$-module locally free of finite type satisfying the following properties:

(i)  Let  $F_M:\M(G)^{(p)}\ra \M(G)$ and $V_M:\M(G)\ra \M(G)^{(p)}$ be the $\cO_S$-linear maps induced respectively by the Frobenius and the Verschiebung of $G$. We have the following exact sequence:
\[\cdots\ra \M(G)^{(p)}\xra{F_M}\M(G)\xra{V_M}\M(G)^{(p)}\ra\cdots.\]

(ii) There is a connection $\nabla: \M(G)\ra \M(G)\otimes_{\cO_S}\Omega^1_{S/\F_p}$ for which $F_M$ and $V_M$ are   horizontal morphisms.

(iii) We have two canonical filtrations by $\cO_S$-modules on $\M(G)$:
\begin{equation}\label{filt-Hodge}
0\ra \omega_G\ra \M(G)\ra \Lie(G^\vee)\ra 0,
\end{equation}
called the \emph{Hodge filtration} on $\M(G)$, and 
\begin{equation}\label{filt-cong}
0\ra \Lie(G^\vee)^{(p)}\xra{\phi_G} \M(G)\ra\omega^{(p)}_{G}\ra 0,
\end{equation}
called the \emph{conjugate filtration} on $\M(G)$. Moreover, we have the following commutative diagram (cf. \cite[2.3.2 and 2.3.4]{Kz})
\begin{equation}\label{diag-Dieud}
\xymatrix{&0\ar[d]&&0\ar[d]&&0\ar[d]&\\
&\omega^{(p)}_G\ar[d]&&\omega_G\ar[d]\ar[rr]^{\psi_G}&&\omega_G^{(p)}\ar[d]&\\
\ar[r]&\M(G)^{(p)}\ar[rr]^{F_M}\ar[d]&&\M(G)\ar[d]\ar[rr]^{V_M}\ar@^{->>}[urr]&&\M(G)^{(p)}\ar[d]\ar[r]&,\\
&\Lie(G^\vee)^{(p)}\ar[d]\ar@^{(->}[urr]^{\phi_G}\ar[rr]^{\widetilde{\varphi_G}}&&\Lie(G^\vee)\ar[d]&&\Lie(G^\vee)^{(p)}\ar[d]&\\
&0&&0&&0&}
\end{equation}
where  the columns are the Hodge filtrations and the anti-diagonal is the conjugate filtration. By functoriality, we see easily that  $\widetilde{\varphi_G}$ above is nothing but the linearization of the Hasse-Witt map $\varphi_G$ \eqref{BT-HW}, and the morphism $\psi_G^*:\Lie(G)^{(p)}\ra \Lie(G)$, which is obtained by applying the functor $\cHom_{\cO_S}(\_,\cO_S)$ to $\psi_G$, is identified  with the linearization $\widetilde{\varphi_{G^\vee}}$ of  $\varphi_{G^\vee}$.

The formation of these structures on $\M(G)$ commutes with arbitrary base changes of $S$. 
In the sequel, we  will use $(\M(G),F_M,\nabla)$ to emphasize these structures on $\M(G)$.

\subsection{}\label{versal} In the reminder of this section, $k$ will denote an algebraically closed field of characteristic $p>0$. Let  $S$ be a  scheme formally smooth over $k$ such that $\Omega^1_{S/\F_p}=\Omega^1_{S/k}$ is an  $\cO_S$-module locally free of finite type, \emph{e.g.} $S=\Spec(A)$ with $A$ a formally smooth $k$-algebra with a finite $p$-basis over $k$. Let $G$ be a BT-group over $S$. We put $\KS$ to be the composed  morphism 
\begin{equation}\label{morph-KS}
\KS: \omega_{G}\ra \M(G)\xra{\nabla}\M(G)\otimes_{\cO_S}\Omega^1_{S/k}\xra{pr}\Lie(G^\vee)\otimes_{\cO_{S}}\Omega^1_{S/k}
\end{equation}
which is $\cO_{S}$-linear. We put  $\cT_{S/k}=\cHom_{\cO_S}(\Omega^1_{S/k},\cO_S)$, and define the \emph{Kodaira-Spencer map} of $G$ 
\begin{equation}\label{Kod-map}
\Kod: \cT_{S/k}\ra \cHom_{\cO_S}(\omega_G,\Lie(G^\vee))
\end{equation}
 to be the morphism induced by $\KS$. We say that $G$ is \emph{versal} if $\Kod$ is surjective.

\subsection{}\label{versal-formal} Let $r$ be an integer $\geq 1$,  $R=k[[t_1,\cdots,t_r]]$, $\m$ be the maximal ideal of $R$. We put $\cS=\Spf(R)$, $S=\Spec(R)$, and for each integer $n\geq 0$, $S_n=\Spec (R/\m^{n+1})$.  By a { BT-group}
$\cG$ over the formal scheme $\cS$, we mean a sequence
 of BT-groups $(G_n)_{n\geq 0}$ over $(S_n)_{n\geq 0}$ equipped with
isomorphisms $G_{n+1}\times_{S_{n+1}}S_n\simeq G_n$.

According to (\cite{dJ} 2.4.4), \emph{the functor $G\mapsto
(G\times_S S_n)_{n\geq0}$ induces an equivalence of categories
between the category of BT-groups over $S$ and the category of
BT-groups over $\cS$.} For a BT-group  $\cG$  over $\cS$, the
corresponding BT-group $G$ over $S$ is called the
\emph{algebraization }of $\cG$. We say that $\cG$ is \emph{versal} over $\cS$, if its algebraization $G$ is versal over $S$. Since $S$ is local, by Nakayama's Lemma,  $\cG$ or $G$ is versal if and only if the reduction of $\Kod$ modulo the maximal ideal
\begin{equation}\label{Kod-0}\Kod_0:\cT_{S/k}\otimes_{\cO_S} k\lra \Hom_{k}(\omega_{G_0},\Lie(G^\vee_0))\end{equation}
is surjective.

\subsection{} We recall briefly  the deformation theory  of a BT-group.
Let $\AL_k$ be the category of local artinian $k$-algebras with
residue field $k$.  We notice that all
morphisms  of $\AL_k$ are  local. A morphism  $A'\ra A$ in $\AL_k$  is
called a \emph{small extension}, if it is surjective and its kernel
$I$ satisfies $I\cdot \m_{A'}=0$, where $\m_{A'}$ is  the maximal ideal of $A'$. 

Let $G_0$ be a BT-group over $k$, and $A$  an object of $\AL_k$.  A
deformation of $G_0$ over $A$ is a pair $(G,\phi)$, where $G$ is a
BT-group over $\Spec (A)$ and $\phi$ is an isomorphism
$\phi:G\otimes_Ak\xra{\sim} G_0$. When there is no risk of
confusions, we will denote a deformation $(G,\phi)$ simply by $G$. Two deformations $(G,\phi)$
and $(G',\phi')$ over $A$ are isomorphic if there exists an
isomorphism of BT-groups $\psi:G\xra{\sim} G'$ over $A$ such that
$\phi=\phi'\circ(\psi\otimes_A k)$. Let's denote by $\D$ the functor
which associates with each object $A$ of $\AL_k$ the set of
isomorphic classes of deformations of $G_0$ over $A$. If  $f:A\ra
B$ is a morphism of $\AL_k$, then the map $\D(f):\D(A)\ra \D(B)$ is
given by extension of scalars. We call $\D$ the \emph{deformation
functor} of $G_0$ over $\AL_k$.

\begin{prop}[\cite{Il} 4.8]\label{prop-deform}  Let $G_0$ be a BT-group over $k$ of dimension $d$ and height
$c+d$, $\D$ be the deformation functor of $G_0$ over
$\AL_k$.

\emph{(i)} Let $A'\ra A$ be a small extension in $\AL_k$ with ideal $I$, $x=(G,\phi)$ be an element in  $\D(A)$, $\D_x(A')$ be the subset of $\D(A')$ with image $x$ in $\D(A)$.  Then the set $\D_x(A')$ is  a nonempty homogenous space under the group  $\Hom_k(\omega_{G_0},\Lie(G_0^\vee))\otimes_k I$.

\emph{(ii)} The functor $\D$ is pro-representable by a formally
smooth formal scheme $\cS$ over $k$ of relative dimension $cd$,
\ie $\cS=\Spf(R)$ with $R\simeq k[[(t_{ij})_{1\leq i \leq c, 1\leq
j\leq d}]]$, and there exists a unique deformation $(\cG,\psi)$ of $G_0$ over $\cS$ such that,
for any object $A$ of $\AL_k$ and any deformation $(G,\phi)$ of
$G_0$ over $A$, there is a unique homomorphism of local $k$-algebras
$\varphi:R\ra A$ with $(G,\phi)=\D(\varphi)(\cG,\psi)$.

\emph{(iii)} Let $\cT_{\cS/k}(0)=\cT_{\cS/k}\otimes_{\cO_{\cS}} k$  be the
tangent space of $\cS$ at its unique closed point,
\begin{equation*}\Kod_0: \cT_{\cS/k}(0)\lra \Hom_k(\omega_{G_0},\Lie(G_0^\vee))
\end{equation*} be the Kodaira-Spencer map  of $\cG$
evaluated at the closed point of $\cS$. Then $\Kod_0$ is bijective,
and it can  be described
 as follows. For an element $f\in \cT_{\cS/k}(0)$, \ie a homomorphism of
 local $k$-algebras
$f:R\ra k[\epsilon]/\epsilon^2$, $\Kod_0(f)$ is the difference of
deformations \[[\cG\otimes_R(k[\epsilon]/\epsilon^2)]-[G_0\otimes_{
k} (k[\epsilon]/\epsilon^2)],\] which is a well-defined element in $
\Hom_k(\omega_{G_0},\Lie(G^\vee_0))$ by \emph{(i)}.
\end{prop}

\begin{rem}\label{rem-basis} Let $(e_j)_{1\leq j \leq d}$ be a basis of $\omega_{G_0}$,
$(f_i)_{1\leq i\leq c}$ be a basis of $\Lie(G_0^\vee)$. In view of
\ref{prop-deform}(iii), we can choose a system of  parameters
$(t_{ij})_{1\leq i \leq c, 1\leq j\leq d}$ of $\cS$ such that
$$\Kod_0(\frac{\partial\empty}{\partial t_{ij}})=e_j^*\otimes f_i,$$
where $(e_j^*)_{1\leq j\leq d}$ is the dual basis of $(e_j)_{1\leq
j\leq d}$. Moreover, if $\m$ is the maximal ideal of $R$, the
parameters $t_{ij}$ are determined uniquely modulo $\m^2$.
\end{rem}

\begin{cor}[\textbf{Algebraization of the universal deformation}]\label{cor-alg-univ}
 The assumptions being those of
$\eqref{prop-deform}$, we put moreover $\bS=\Spec(R)$ and $\bG$ the
algebraization of the universal formal deformation $\cG$. Then the BT-group $\bG$ is  versal over $\bS$, and satisfies the following universal property: Let
$A$ be a noetherian complete local $k$-algebra with residue field
$k$, $G$ be a BT-group over $A$ endowed with an isomorphism
$G\otimes_Ak\simeq G_0$. Then there exists a unique continuous
homomorphism of local $k$-algebras $\varphi:R\ra A$ such that
$G\simeq \bG\otimes_RA$.
\end{cor}

\begin{proof} By the last remark of \ref{versal-formal}, $\bG$ is clearly versal. It remains to prove that it satisfies the universal property in the corollary. Let $G$ be a deformation of $G_0$ over a
noetherian complete local $k$-algebra $A$ with residue field $k$. We
denote by $\m_A$ the maximal ideal of $A$, and put
$A_n=A/\m_A^{n+1}$ for each integer $n\geq0$. Then by
\ref{prop-deform}(b), there exists a unique local homomorphism
$\varphi_n:R\ra A_n$ such that $G\otimes A_n\simeq \bG\otimes_RA_n$.
The $\varphi_n$'s form a projective system $(\varphi_n)_{n\geq0}$,
whose projective limit $\varphi:R\ra A$ answers the question.
\end{proof}

\begin{defn}\label{defn-moduli}
The notations are those of \eqref{cor-alg-univ}. We  call $\bS$ the
\emph{local moduli  in characteristic $p$} of $G_0$, and $\bG$ the
\emph{universal deformation of $G_0$ in characteristic $p$}. 
\end{defn}
If there is no confusions, we will
omit ``in characteristic $p$'' for short.

\subsection{} Let $G$ be a BT-group over $k$, $G^\circ$ be its connected part, and $G^\et$ be its \'etale part. Let $r$ be the height of $G^\et$. Then we have $G^\et\simeq (\Q_p/\Z_p)^r$, since $k$ is algebraically closed.    Let $\D_{G}$ (\resp $\D_{G^{\circ}}$) be the deformation functor   of $G$ (\resp $G^\circ$) over $\AL_k$. If $A$ is an object in $\AL_k$ and $\cG$ is a deformation  of $G$ (\resp $G^\circ$) over $A$, we denote by $[\cG]$ its isomorphic class in $\D_{G}(A)$ (\resp in $\D_{G^\circ(A)}$).

\begin{prop}\label{prop-def-surj} The assumptions are as above, let 
$\Theta: \D_{G}\ra \D_{G^{\circ}}$ be the morphism of functors that maps a deformation of $G$ 
to its connected component. 

 \emph{(i)} The morphism $\Theta$ is formally smooth of relative dimension $r$.

\emph{(ii)} Let $A$ be an object of $\AL_k$, and $\cG^\circ$ be a deformation of $G^\circ$ over $A$. Then the subset  $\Theta_{A}^{-1}([\cG^{\circ}])$ of $\D_{G}(A)$ is canonically identified with $\Ext^1_{A}(\Q_p/\Z_p,\cG^\circ)^r$, where $\Ext^1_A$ means the group of extensions in the category of abelian $\fppf$-sheaves on $\Spec(A)$.
\end{prop} 

\begin{proof} (i) Since $\D_{G}$ and $\D_{G^\circ}$ are both
  pro-representable by a noetherian local complete $k$-algebra  and
  formally smooth over $k$ (\ref{prop-deform}),  by a formal
  completion version of [EGA $\mathrm{IV} 17.11.1(d)$], we only need to check that the tangent map $$\Theta_{k[\epsilon]/\epsilon^2}:\D_{G}(k[\epsilon]/\epsilon^2)\ra \D_{G^{\circ}}(k[\epsilon]/\epsilon^2)$$
 is surjective with kernel of dimension $r$ over $k$. By
  \ref{prop-deform}(iii), $\D_{G}(k[\epsilon]/\epsilon^2)$
 (\resp $\D_{G^\circ}(k[\epsilon]/\epsilon^2)$) is
 isomorphic to $\Hom_{k}(\omega_{G},\Lie(G^\vee))$
 (\resp
 $\Hom_{k}(\omega_{G^{\circ}},\Lie(G^{\circ\vee}))$)
 by the Kodaira-Spencer morphism. In view of the canonical isomorphism $\omega_{G}\simeq \omega_{G^{\circ}}$,   $\Theta_{k[\epsilon]/\epsilon^2}$ corresponds  to the map
\[\Theta'_{k[\epsilon]/\epsilon^2}: \Hom_{k}(\omega_{G},\Lie(G^\vee))\ra \Hom_{k}(\omega_{G},\Lie(G^{\circ\vee}))\]
induced by the canonical surjection $\Lie(G^{\vee})\ra \Lie(G^{\circ\vee})$. It is clear that $\Theta'_{k[\epsilon]/\epsilon^2}$ is surjective of kernel $\Hom_{k}(\omega_{G},\Lie(G^{\et\vee}))$, which has dimension $r$ over $k$.

(ii) Since $G^{\et}$ is isomorphic to $(\Q_p/\Z_p)^r$, every element in $\Ext^1_{A}(\Q_p/\Z_p,\cG^{\circ})^r$ defines clearly an element of $\D_{G}(A)$  with image $[\cG^\circ]$ in $\D_{G^{\circ}}(A)$. Conversely, for any $\cG\in \D_{G}(A)$ with connected component isomorphic to $\cG^{\circ}$, the isomorphism $G^{\et}\simeq (\Q_p/\Z_p)^r$ lifts uniquely to an isomorphism $\cG^\et\simeq (\Q_p/\Z_p)^r$ because $A$ is henselian. The canonical exact sequence $0\ra \cG^{\circ}\ra \cG\ra \cG^{\et}\ra 0$ shows that $\cG$ comes from an element of $\Ext^1_{A}(\Q_p/\Z_p,\cG^{\circ})^r$.

\end{proof}

\section{HW-cyclic Barsotti-Tate Groups}

\begin{defn}\label{defn-cyclic} Let $S$ be a scheme of characteristic
  $p>0$, $G$ be a BT-group over $S$ such that  $c=\dim(G^\vee)$ is constant. 
We say that $G$ is \emph{HW-cyclic}, if $c\geq 1$ and there exists an element $v\in \Gamma(S,\Lie(G^\vee))$ such that 
\[v, \HW_G(v), \cdots, \HW_G^{c-1}(v)\]
generate $\Lie(G^\vee)$ as an $\cO_S$-module, where $\varphi_G$ is the Hasse-Witt map \eqref{BT-HW} of $G$.
\end{defn}

\begin{rem}
 It is clear that a BT-group $G$ over $S$ is HW-cyclic, if and only if $\Lie(G^\vee)$ is free over $\cO_S$ and there exists a  basis of $\Lie(G^\vee)$ over $\cO_S$ under which  
$\varphi_{{G}}$ is expressed  by a matrix of the form
\begin{equation}\label{HW-matrix}
\begin{pmatrix}0 &0 &\cdots &0 &-a_1\\
1 &0 &\cdots &0 &-a_2\\
0 &1 &\cdots &0 &-a_3\\
\vdots &&\ddots &&\vdots\\
0&0&\cdots&1&-a_{c}\end{pmatrix},
\end{equation}
where $a_i\in \Gamma(S,\cO_S)$ for $1\leq i\leq c$.

\end{rem} 

\begin{lemma}\label{lemma-cyclic} Let $R$ be a local ring of characteristic $p>0$, $k$ be its residue field. 

\emph{(i)} A BT-group $G$ over $R$ is HW-cyclic if and only if so is $G\otimes k$. 

\emph{(ii)} Let $0\ra G'\ra G\ra G''\ra 0$ be an exact sequence of
BT-groups over $R$. If $G$ is HW-cyclic, then so is $G'$. In
particular, if $R$ is henselian, the connected part of a HW-cyclic BT-group over $R$ is HW-cyclic.
\end{lemma}
\begin{proof} 
(i) The property of being HW-cyclic is clearly stable under arbitrary
  base changes, so the ``only if'' part is clear. Assume that
  $G_0=G\otimes k$ is HW-cyclic. Let $\overline{v}$ be an element of
  $\Lie(G^\vee_0)=\Lie(G^\vee)\otimes k$ such that
  $(\overline{v},\varphi_{G_0}(\overline{v}),\cdots,
  \varphi^{c-1}_{G_0}(\overline{v}))$ is a basis of
  $\Lie(G^\vee_0)$. Let $v$ be any lift of $\overline{v}$ in
  $\Lie(G^\vee)$. Then by Nakayama's lemma, $(v, \varphi_G(v),\cdots,
  \varphi_G^{c-1}(v))$ is a basis of $\Lie(G^\vee)$. 

(ii) By statement (i), we may assume $R=k$. The exact sequence of BT-groups induces an exact sequence of
Lie algebras
\begin{equation}\label{exseq-Lie-alg}
0\ra\Lie(G''^\vee)\ra\Lie(G^\vee)\ra \Lie(G'^\vee)\ra 0,
\end{equation}
 and the Hasse-Witt map $\HW_{G'}$ is
induced by $\HW_G$ by functoriality. Assume that $G$ is  HW-cyclic
and $G^\vee$ has  dimension $c$. Let $u$ be an
element of $\Lie(G^\vee)$ such that $${u},
\HW_{G}({u}), \cdots, \HW_{G}^{c-1}({u})$$
form a basis of $\Lie(G^\vee)$ over $k$. We denote by $u'$ the
image of $u$ in $\Lie(G'^{\vee})$.
 Let $r\leq c$ be the maximal integer  such
that the vectors 
$$
 u',\;\,
\HW_{G'}(u'),\;\, \cdots,\;\,
\HW_{G'}^{r-1}(u')
$$ are linearly
independent over $k$. It is easy to see that they form a basis of  the $k$-vector
space $\Lie(G'^{\vee})$. Hence $G'$ is HW-cyclic.
\end{proof}

\begin{lemma}\label{lemma-HW-V} Let  $S=\Spec(R)$ be an affine scheme
  of characteristic $p>0$,  $G$ be a HW-cyclic BT-group over $R$ with
  $c=\dim(G^\vee)$ constant, and 
\[\begin{pmatrix}0 &0 &\cdots &0 &-a_1\\
1 &0 &\cdots &0 &-a_2\\
0 &1 &\cdots &0 &-a_3\\
\vdots &&\ddots &&\vdots\\
0&0&\cdots&1&-a_{c}\end{pmatrix}\in \rM_{c\times c}(R),\] be a matrix
of $\varphi_G$. Put $a_{c+1}=1$, and $P(X)=\sum_{i=0}^c
a_{i+1}X^{p^i}\in R[X]$.

\emph{(i)} Let $V_G:G^{(p)}\ra G$ be the Verschiebung homomorphism of $G$.
Then $\Ker V_G$ is isomorphic to the group scheme $\Spec(R[X]/P(X))$ with comultiplication given by $X\mapsto 1\otimes X+ X\otimes 1$.

\emph{(ii)} Let  $x\in S$, and $G_x$ be the fibre of $G$ at $x$.  Put
\begin{equation}\label{integer-i_0}
i_0(x)=\min_{0\leq i\leq c}\{i; a_{i+1}(x)\neq 0\},
\end{equation}
where $a_i(x)$ denotes the image of $a_i$ in the residue field of $x$. Then the \'etale part of $G_x$ has height $c-i_0(x)$, and the connected part of $G_x$ has height $d+i_0(x)$. In particular, $G_x$ is connected if and only if $a_i(x)=0$ for $1\leq i\leq c$.
\end{lemma}

\begin{proof}(i) By \ref{GV-pLie} and \ref{lemma-KerV}, $\Ker V_G$ is isomorphic to the group scheme 
\begin{equation*}\Spec\biggl(R[X_1,\dots,X_c]/(X_1^p-X_2,\cdots, X_{c-1}^p-X_c,X_c^p+a_1X_1+\cdots+a_{c}X_c)\biggr)\end{equation*}
with comultiplication $\Delta(X_i)=1\otimes X_i+X_i\otimes 1$ for $1\leq i \leq c$. By sending $(X_1,X_2,\cdots, X_c)\mapsto (X,X^p,\cdots, X^{p^{c-1}})$, we see that the above group scheme is isomorphic to $\Spec(R[X]/P(X))$ with comultiplication $\Delta(X)=1\otimes X+X\otimes 1$.

(ii) By base change, we may assume that $S=x=\Spec(k)$ and hence $G=G_x$. Let $G(1)$ be the kernel of the multiplication by $p$ on $G$. Then we have an exact sequence
\[0\ra\Ker F_G\ra G(1)\ra \Ker V_G\ra 0. \]
Since $\Ker F_G$ is an infinitesimal group scheme over $k$, we have $G(1)(\kb)=(\Ker V_G)(\kb)$, where $\kb$ is an algebraic closure of $k$. 
By the definition of $i_0(x)$, we have $P(X)=Q(X^{p^{i_0(x)}})$, where $Q(X)$ is an additive sepearable polynomial in $k[X]$ with $\deg(Q)=p^{c-i_0(x)}$. Hence the roots of $P(X)$ in $\kb$ form an $\F_p$-vector space of dimension $c-i_0(x)$. By (i),  $(\Ker V_G)(\kb)$ can be identified with the additive group consisting of the roots of $P(X)$ in $\kb$. Therefore, the  \'etale part of $G$ has height $c-i_0(x)$, and the connected part of $G$ has height $d+i_0(x)$.
\end{proof}

\subsection{}\label{sect-a-num} Let $k$ be a perfect field of
characteristic $p>0$, and $\alpha_p=\Spec(k[X]/X^p)$ be the finite group scheme over $k$ with comultiplication map $\Delta(X)=1\otimes X+X\otimes 1$. Let $G$ be a BT-group over $k$. Following Oort,  we call 
\[a(G)=\dim_{k} \Hom_{k_{\fppf}}(\alpha_p,G)\]
the $a$-number of $G$, 
where $\Hom_{k_{\fppf}}$ means the homomorphisms in the category of abelian $\fppf$-sheaves over $k$. Since the Frobenius of $\alpha_p$ vanishes, any morphism of $\alpha_p$ in $G$ factorize through $\Ker(F_G)$. Therefore we have 
\begin{align*}\Hom_{k_{\fppf}}(\alpha_p,G)&=\Hom_{k-gr}(\alpha_p,\Ker(F_G))\\
&=\Hom_{k-gr}(\Ker(F_G)^\vee,\alpha_p)\\
&=\Hom_{\pLie_k}(\Lie(\alpha_p),\Lie(\Ker(F_G))),
\end{align*}
where $\Hom_{k-gr}$ denotes the homomorphisms in the category of commutative group schemes over $k$, and the last equality uses Proposition \ref{GV-pLie}. Since we have a canonical isomorphism $\Lie(\Ker(F_G))\simeq \Lie(G)$  and $\Lie(\alpha_p)$ has dimension one over $k$ with  $\varphi_{\alpha_p}=0$,  we get
\begin{equation}\label{a-number}a(G)=\dim_k\{x\in \Lie(G)| \varphi_{G^\vee}(x)=0\}=\dim_{k}\Ker(\varphi_{G^\vee}).\end{equation}
Due to the perfectness of $k$, we have also $a(G)=\dim_k\Ker
(\widetilde{\varphi_{G^\vee}})$, where $ \widetilde{\varphi_{G^\vee}}$ is the linearization of $\varphi_{G^\vee}$.
By Proposition \ref{prop-ord}, we see that $a(G)=0$ if and only if $G$ is ordinary.

\begin{lemma}\label{lemma-a-number}Let $G$ be a BT-group over $k$, and $G^\vee$ its Serre dual. Then we have $a(G)=a(G^\vee)$.
\end{lemma}
\begin{proof} Let  $\psi_G:\omega_G\ra \omega_G^{(p)}$ be the $k$-linear map induced by the Verschiebung of $G$. Then $\psi^*_G$, the morphism obtained by applying the functor $\Hom_k(\_,k)$ to $\psi_G$, is identified with $\widetilde{\varphi_{G^\vee}}$. By \eqref{a-number} and the exactitude of the functor $\Hom_k(\_,k)$, we have $a(G)=\dim_k\Ker(\psi^*_G)=\dim_k\Coker(\psi_G)$. Using the additivity of $\dim_k$, we get finally $a(G)=\dim_k\Ker(\psi_G)$. By considering the commutative diagram \eqref{diag-Dieud}, we have 
\[a(G)=\dim_k \biggl(\omega_G\cap \phi_G(\Lie(G^\vee)^{(p)})\biggr).\]
On the other hand,  it follows also from \eqref{diag-Dieud} that 
\[a(G^\vee)=\dim_k\Ker(\widetilde{\varphi_G})=\dim_{k}\biggl(\phi_G(\Lie(G^\vee)^{(p)})\cap \omega_G\biggr).\]
The lemma now follows immediately.

\end{proof}

\begin{prop}\label{prop-HW-a} Let $k$ be a perfect field of characteristic $p>0$, $G$  a BT-group over $k$.
Consider the following conditions:

\emph{(i)} $G$ is HW-cyclic and non-ordinary;

\emph{(ii)} the connected part $G^\circ$ of $G$ is HW-cyclic and not of multiplicative type;

\emph{(iii)} $a(G^\vee)=a(G)=1$.

We have $\mathrm{(i)}\Rightarrow \mathrm{(ii)}\Leftrightarrow \mathrm{(iii)}$. If $k$ is algebraically closed, we have moreover $\mathrm{(ii)}\Rightarrow \mathrm{(i)}$.
\end{prop}
\begin{rem} In \cite[Lemma 2.2]{oort}, Oort proved the following assertion, which is a generalization of $\mathrm{(iii)}\Rightarrow\mathrm{(ii)}$: Let  $k$ be an algebraically closed field of characteristic $p>0$, and $G$ be a connected BT-group with $a(G)=1$. Then  there exists a basis of the Dieudonn\'e module $M$ of $G$ over $W(k)$, such that the action of Frobenius on $M$ is given by a display-matrix of ``normal form'' in the sense of \cite[2.1]{oort}.
\end{rem}

\begin{proof} $\mathrm{(i)}\Rightarrow \mathrm{(ii)}$ follows from \ref{lemma-cyclic}(ii).

$\mathrm{(ii)}\Rightarrow\mathrm{(iii)}$. First, we note that $a(G)=a(G^\circ)$, so we may assume $G$ connected. Since $G$ is not of multiplicative type, we have $c=\dim(G^\vee)\geq 1$.  By Lemma \ref{lemma-HW-V}(ii), there exists a basis of $\Lie(G^\vee)$ over $k$ under which $\varphi_G$ is expressed by 
\[\begin{pmatrix}0 &0 &\cdots &0 &0\\
1 &0 &\cdots &0 &0\\
0&1&\cdots &0&0\\
\vdots &&\ddots &&\vdots\\
0&0&\cdots&1&0\end{pmatrix}\in \rM_{c\times c}(k).\] According to \eqref{a-number}, $a(G^\vee)$ equals to $\dim_k\Ker(\varphi_G)$, \ie the $k$-dimension of the solutions of the equation system in $(x_1,\cdots, x_c)$
\[\begin{pmatrix}0 &0 &\cdots &0 &0\\
1 &0 &\cdots &0 &0\\
\vdots &&\ddots &&\vdots\\
0&0&\cdots&1&0\end{pmatrix}\begin{pmatrix}x_1^p\\
x_2^p\\ 
\vdots\\x_c^p\end{pmatrix}=0\]
The solutions $(x_1,\cdots,x_c)$ form clearly a vector space  over $k$ of dimension $1$, \ie we have $a(G^\vee)=1$.

$\mathrm{(iii)}\Rightarrow\mathrm{(ii)}.$ Let  $G^\et$ be the \'etale part of $G$. Since $k$ is perfect, the exact sequence \eqref{decomp-BT} splits \cite[Chap. II \S7]{De}; so we have $G\simeq G^\circ\times G^\et$. We put $M=\Lie(G^\vee)$, $M_1=\Lie(G^{\circ\vee})$ and $M_2=\Lie(G^{\et\vee})$ for short. By \ref{prop-etale-HW} and \ref{cor-nilp-HW}, we have a decomposition $M=M_1\oplus M_2$, such that $M_1,M_2$ are stable under $\varphi_G$, and the action of $\varphi_G$ is nilpotent on $M_1$ and bijective on $M_2$. We note that $a(G^{\circ\vee})=a(G^\circ)=a(G)=1$. By the last remark of \ref{sect-a-num}, $G^\circ$ is not of multiplicative type, hence $\dim_kM_1=\dim(G^{\circ \vee})\geq 1$. It remains to prove that $G^\circ$ is HW-cyclic.
 Let $n$ be the minimal integer such that $\varphi^n_G(M_1)=0$. We have a strictly increasing filtration 
\[0\subsetneq \Ker(\varphi_G)\subsetneq \cdots \subsetneq \Ker(\varphi_G^n)=M_1.\]
If $n=1$, then $M_1$ is one-dimensional, hence $G^\circ$ is clearly HW-cyclic. Assume $n\geq 2$.
For $2\leq m\leq n$,  $\varphi_G^{m-1}$ induces an injective map 
\[\overline{\varphi_G^{m-1}}:\Ker(\varphi_G^{m})/\Ker(\varphi_G^{m-1})\lra \Ker(\varphi_G).\]
Since  $\dim_k\Ker (\varphi_G)=a(G^{\circ\vee})=1$, $\overline{\varphi_G^{m-1}}$ is
necessarily bijective. So we have $\dim_k \Ker(\varphi_G^{m})=m$ for $1\leq m \leq n$. Let $v$ be an element of $M_1$ but not in $\Ker(\varphi_G^{n-1})$. Then $v,\varphi_G(v),\cdots,\varphi_G^{n-1}(v)$ are linearly independant, hence they form a basis of $M_1$ over $k$. This proves that $G^\circ$ is HW-cyclic.

Assume  $k$ algebraically closed. We  prove that $\mathrm{(ii)}\Rightarrow \mathrm{(i)}$. Noting that $G$ is ordinary if and only if $G^\circ$ is  of multiplicative type, we only need to check that $G$ is HW-cyclic.  We conserve the notations above.  Since $\varphi_G$ is bijective on $M_2$ and $k$ algebraically closed,  there exists a basis $(e_1,\cdots, e_m)$ of $M_2$ such that $\varphi_G(e_i)=e_i$ for $1\leq i \leq m$. Let $v\in M_1$ but not in $\Ker(\varphi_G^{n-1})$ as above, and put $u=v+\lambda_1 e_1+\cdots \lambda_m e_m$, where $\lambda_i(1\leq i\leq m)$ are some elements in $k$ to be determined later. Then we have 
\[\begin{pmatrix}\varphi_G^{n}(u)\\
\vdots\\ \varphi_G^{n+m-1}(u)\end{pmatrix}=
\begin{pmatrix}\lambda_1^{p^n}& \cdots &\lambda_m^{p^n}\\
\vdots &\ddots & \vdots\\
\lambda_1^{p^{n+m-1}}&\cdots & \lambda_m^{p^{n+m-1}}
\end{pmatrix}\begin{pmatrix}e_1\\ \vdots\\ e_m \end{pmatrix}.\]
Let $L(\lambda_1,\cdots,\lambda_m)\in k[\lambda_1,\cdots, \lambda_m]$
be the determinant polynomial of the matrix on the right side. An
elementary computation shows that the polynomial
$L(\lambda_1,\cdots,\lambda_m)$ is not null.   We can choose
$\lambda_1,\cdots, \lambda_m\in k$ such that
$L(\lambda_1,\cdots,\lambda_m)\neq 0$ because $k$ is algebraically closed. So  $\varphi_G^n(u),\cdots, \varphi_G^{n+m-1}(u)$ form a basis of $M_2$ over $k$. Since 
$$\varphi^{i}_G(u)\equiv \varphi^{i}_G(v) \mod M_2 \quad\text{ for} \quad 0\leq i \leq n,$$ 
by the choice of $u$, we see that $\{u,\varphi_G(u), \cdots, \varphi_G^{n+m-1}(u)\}$ form a basis of $M=\Lie(G^\vee)$ over $k$.
\end{proof}

By combining \ref{lemma-a-number} and \ref{prop-HW-a}, we obtain the following
\begin{cor} Let $k$ be an algebraically closed field of characteristic $p>0$.  Then  a  BT-group over $k$ is HW-cyclic if and only if so is its Serre dual. 
\end{cor}

\subsection{Examples}\label{HW-exem} Let $k$ be a perfect field,
$W(k)$ be the ring of Witt vectors with coefficients in $k$,  and $\sigma$ be the Frobenius automorphism of $W(k)$. Let $s,r$ be relatively prime integers such that $0\leq s\leq r$ and $r\neq 0$; put $\lambda=\frac{s}{r}$. We consider the Dieudonn\'e module $M^\lambda\simeq W(k)[F,V]/(F^{r-s}-V^s)$, where $W(k)[F,V]$ is the non-commutative ring with relations $FV=VF=p$, $Fa=\sigma(a)F$ and $V\sigma(a)=aV$ for all  $a\in W(k)$.
We note that $M^\lambda$ is free of rank $r$ over $W(k)$  and
$M^\lambda/VM^\lambda\simeq k[F]/F^{r-s}$. By the contravariant Dieudonn\'e theory, $M^\lambda$ corresponds to a BT-group $G^\lambda$ over $k$ of height $r$ with $\Lie(G^{\lambda\vee})=M^\lambda/VM^\lambda$. We see easily that $G^\lambda$ is HW-cyclic, and we call it the \emph{elementary BT-group of slope $\lambda$.} We note that $G^0\simeq \Q_p/\Z_p$, $G^1\simeq \mu_{p^\infty}$, and $(G^{\lambda})^\vee\simeq G^{1-\lambda}$ for $0\leq \lambda\leq 1$. 

Assume $k$  algebraically closed. Then by the Dieudonn\'e-Manin's classification of isocrystals \cite[Chap.IV \S4]{De}, any BT-group  over $k$ is isogenous to a finite product of $G^\lambda$'s; moreover, any connected one-dimensional BT-group over $k$ of height $r$ is necessarily isomorphic to $G^{1/r}$ \cite[Chap.IV \S8]{De}, hence  in particular HW-cyclic.

\begin{prop}\label{prop-HW-versal}
 Let $k$ be an algebraically closed field of characteristic $p>0$, $R$
 be a noetherian complete regular local $k$-algebra with residue field
 $k$, and $S=\Spec(R)$. Let $G$ be a connected HW-cyclic BT-group over $R$ of dimension $d\geq 1$ and height $c+d$,
\[\h=\begin{pmatrix}0 &0 &\cdots &0 &-a_1\\
1 &0 &\cdots &0 &-a_2\\
0&1&\cdots&0&-a_3\\
\vdots &&\ddots &&\vdots\\
0&0&\cdots&1&-a_c\end{pmatrix}\in \rM_{c\times c}(R)\] be a matrix of $\HW_G$.

\emph{(i)} If $G$ is versal over $S$, then $\{a_1,\cdots, a_c\}$ is a subset of a regular system of parameters of $R$.

\emph{(ii)} Assume that $d=1$. The converse of \emph{(i)} is also true, \ie if $\{a_1,\cdots,a_c\}$ is a subset of a regular system of parameters of $R$ then $G$ is versal over $S$. Furthermore, $G$ is the universal deformation of its special
fiber if and only if $\{a_1,\cdots, a_c\}$ is a system
of regular parameters of $R$.  
\end{prop}

\begin{proof}  Let $(\M(G),F_M,\nabla)$ be the finite free $\cO_S$-module equipped with a semi-linear endomorphism $F_M$ and a connection $\nabla:\M(G)\ra \M(G)\otimes_{\cO_{S}}\Omega^1_{S/k}$, obtained by evaluating the Dieudonn\'e crystal of $G$ at the trivial immersion $S\hra S$  (cf. \ref{pre-Dieud}). Recall that 
we have a commutative diagram 
\begin{equation}\label{diag-F-phi}
\xymatrix{\M(G)^{(p)}\ar[rr]^{F_M}\ar[d]_{pr}&&\M(G)\ar[d]^{pr}\\
\Lie(G^\vee)^{(p)}\ar[rr]^{\widetilde{\varphi_G}}\ar@{^(->}[urr]^{\phi_G}&&\Lie(G^\vee),}\end{equation}
where $\phi_G$ is universally injective \eqref{diag-Dieud}.
 Let
$\{v_1,\cdots, v_c\}$  be a basis of $\Lie(G^\vee)$ over $\cO_S$ under
 which $\HW_G$
is expressed by $\h$, \ie we have $\varphi_G^{i-1}(v_1)=v_i$ for
$1\leq i\leq c$ and
$\varphi_G^{c}(v_1)=\varphi_G(v_c)=-\sum_{i=1}^ca_iv_i$. Let $f_1$ be
a lift of $v_1$ to $\Gamma(S,\M(G))$, and put
$f_{i+1}=\phi_G(v_i^{(p)})$ for $1\leq i \leq c-1$, where
$v_i^{(p)}=1\otimes v_i\in\Gamma(S,\Lie(G^\vee)^{(p)})$.  The image
of $f_i$ in $\Gamma(S,\Lie(G^\vee))$ is thus $v_i$ for $1\leq i\leq c$ by \eqref{diag-F-phi}.  We put 
\begin{equation}\label{defn-e1}
e_1=\phi_G(v_c^{(p)})+a_1f_1+\cdots +a_cf_c\in \Gamma(S,\M(G)).
\end{equation}
 The image of  $e_1$ in $\Gamma(S,\Lie(G^\vee))$ is
 $\varphi_G(v_c)+\sum_{i=1}^ca_iv_i=0$; so we have $e_1\in
 \Gamma(S,\omega_G).$  By \ref{lemma-HW-V}(ii), we  notice  that  $a_1,\cdots,
a_c$ belong to the maximal ideal $\m_R$ of $R$, as $G$ is
connected. Hence, we have $\overline{e_1}=\overline{\phi_G(v^{(p)}_c)}$, where for a $R$-module $M$ and $x\in M$, we denote by $\xb$ the
canonical image of $x$ in $M\otimes k$. Since  $\phi_G$ commutes with
base change  and  is universally
injective, we get
$\overline{e_1}=\overline{\phi_G(v^{(p)}_c)}=\phi_{G\otimes k}(\overline{v^{(p)}_c})\neq 0$. Therefore, we can choose
 $e_2,\cdots, e_d\in \Gamma(S,\omega_G)$ such that $(e_1,\cdots, e_d)$
 becomes a basis of $\omega_G$ over $\cO_S$, so  $(e_1,\cdots,e_d,f_1,\cdots, f_c)$ is a basis of $\M(G)$. 
Since $F_M$ is horizontal for the connection $\nabla$ (cf. \ref{pre-Dieud}(ii)), we have  \[\nabla(\phi_G(v^{(p)}_c))=\nabla(F_M(f_c^{(p)}))=0.\]  In view of \eqref{defn-e1}, we get
\begin{align}\nabla(e_1)&=\sum_{i=1}^c f_i\otimes da_i+\sum_{i=1}^c
a_i\nabla(f_i)\nonumber\\
&\equiv \sum_{i=1}^c f_i\otimes da_i \quad (\mathrm{mod}\; \m_R).\label{nabla-e_1}
\end{align} 
Let $\KS_0$ and $\Kod_0$ be respectively the reductions modulo $\m_R$ of \eqref{morph-KS} and \eqref{Kod-map}. Since $(\overline{v_i})_{1\leq i\leq c}$ is a base of $\Lie(G^\vee)\otimes k$, we can write 
\[\KS_0(e_j)=\sum_{i=1}^c\overline{v_i}\otimes \theta_{i,j} \quad \quad \text{for $1\leq j\leq d$,}\]
 where $\theta_{i,j}\in \Omega_{S/k}\otimes k$.
From \eqref{nabla-e_1}, we deduce that $\theta_{i,1}=da_i$.
By the definition of $\Kod_0$, we have
\begin{equation}\label{equ-Kod_0}
\Kod_0(\partial)=\sum_{j=1}^d\sum_{i=1}^c<\partial,\theta_{i,j}>\overline{e_j}^*\otimes \overline{v_i}
\end{equation}
where 
${\partial}\in \cT_{S/k}\otimes k$,  $<\bullet,\bullet>$ is the canonical pairing between
$\cT_{S/k}\otimes k$ and $\Omega^1_{S/k}\otimes k$, and $(\overline{e_i}^*)_{1\leq i\leq d}$ denotes the dual basis of $(\overline{e_i})_{1\leq i\leq d}$. Now assume that $G$ is versal over $S$, \ie  $\Kod_0$ is surjective by definition \eqref{versal}. In particular, there are $\partial_1,\cdots, \partial_c\in \cT_{S/k}\otimes k$ such that $\Kod_0(\partial_i)=\overline{e_1}^*\otimes v_i$ for $1\leq i\leq c$, \ie   we have 
\begin{equation}\label{formula-versal-1}<\partial_i,da_j>=\begin{cases}1 & \text{if $i=j$}\\
0& \text{if $i\neq j$} \end{cases}\quad \text{for $1\leq i,j\leq c$,}\end{equation}
and \[ <\partial_i, \theta_{\ell,j}>=0 \quad \quad  \text{for $1\leq i,j \leq c, 2\leq\ell\leq d $}.\] 
From \eqref{formula-versal-1}, we see easily that $da_1,\cdots, da_c$ are linearly independent in $\Omega_{S/k}\otimes k\simeq \m_R/\m_R^2$; therefore, $(a_1,\cdots, a_c)$ is a part of a regular system of parameters of $R$. Statement (i) is proved.

  For statement (ii), we assume $d=1$ and  that $(a_1,\cdots,a_c)$ is a part of a regular system of parameters of $R$. Then the formula \eqref{equ-Kod_0} is simplified as
\[\Kod_0(\partial)=\sum_{i=1}^c<\partial,da_i>\overline{e_1}^*\otimes \overline{v_i}.\]
 Since $da_1,\cdots, da_c$ are linearly independent in $\Omega_{S/k}^1\otimes k$, there exist $\partial_1,\cdots, \partial_c\in \cT_{S/k}\otimes k$ such that \eqref{formula-versal-1} holds, \ie $(\overline{e_1}^*\otimes \overline{v_i})_{1\leq i\leq c}$ are in the image of $\Kod_0$. But the elements $(\overline{e_1}^*\otimes\overline{v_i})_{1\leq i\leq c}$ form already a basis of $\cHom_{\cO_S}(\omega_G,\Lie(G^\vee))\otimes k$. So $\Kod_0$ is surjective, and hence $G$ is versal over $S$ by Nakayama's lemma. Let $G_0$ be the special fiber of $G$. It remains to prove that when $d=1$, $G$ is the universal
deformation of $G_0$ if and only if 
$\dim(S)=c$ and $G$ is versal over $S$. Let $\bS$ be the local moduli in characteristic $p$ of $G_0$. By the universal property of $\bG$ \eqref{cor-alg-univ}, there exists a unique morphism 
$f:S\ra \bS$ such that  $G\simeq \bG\times_{\bS}S$. Since $S$ and $\bS$ are local complete regular schemes
over $k$ with residue field $k$ of  the same dimension, $f$ is an
isomorphism if and only if the tangent map of $f$ at the closed point
of $S$, denoted by $T_f$, is an isomorphism. By the functoriality of Kodaira-Spencer maps \eqref{Kod-map}, we have a commutative diagram
\[\xymatrix{\cT_{S/k}\otimes_{\cO_S} k\ar[d]_{T_f}\ar[rr]^{\Kod_0^S}&&\Hom_k(\omega_{G_0},\Lie(G_0^\vee))\ar@{=}[d]\\
\cT_{\bS/k}\otimes_{\cO_{\bS}}k\ar[rr]^{\Kod_0^{\bS}} &&\Hom_k(\omega_{G_0},\Lie(G_0^\vee))},\]
where horizontal arrows are the Kodaira-Spencer maps evaluated at the closed points \eqref{Kod-0}. Since $\Kod_0^S$ and $\Kod_0^{\bS}$ are isomorphisms according to the first part of this propostion, we deduce that so is $T_f$. This completes the proof.
\end{proof}

\section{Monodromy of a HW-cyclic BT-group over a Complete Trait of Characteristic $p>0$}

\subsection{}\label{nota-dvr}
Let $k$ be an algebraically closed field of \car $p>0$,   $A$ be a
complete discrete valuation ring of \car $p$, with residue field $k$
and fraction field $K$.  We put  $S=\Spec(A)$, and denote by $s$ its closed
point, by $\eta$ its generic point. Let $\Kb$ be an algebraic closure of $K$, $\Ks$
be the maximal separable extension of $K$ contained in $\Kb$,
$K^{\mathrm{t}}$ be the maximal tamely ramified extension of $K$
contained in $\Ks$. We put $I=\Gal(\Ks/K)$, $I_p=\Gal(\Ks/K^\tame)$
and $I_t=I/I_p=\Gal(K^\tame/K)$.

Let $\pi$ be a uniformizer of $A$; so we have $A\simeq k[[\pi]]$.
Let $\tv$ be the valuation on $K$ normalized by $\tv(\pi)=1$; we
denote also by $\tv$ the unique extension of $\tv$ to $\Kb$.  For
every $\alpha\in \Q$, we denote by $ \m_{\alpha}$ (\resp by
$\m_{\alpha}^+$) the set of elements $x\in \Ks$ such that
$\tv(x)\geq \alpha$ (\resp $\tv(x)> \alpha$). We put
\begin{equation}\label{defn-V_alpha}V_\alpha=\m_\alpha/\m_\alpha^+,\end{equation}
 which is a $k$-vector space of dimension 1 equipped with a continuous action of  the Galois group
$I$.

\subsection{}\label{galois-char} First, we recall some properties of the inertia groups
 $I_p$ and
$I_t$  \cite[Chap. IV]{CL}. The subgroup $I_p$, called the
\emph{wild inertia subgroup}, is the unique maximal pro-$p$-group contained
in $I$ and  hence normal in $I$. The quotient $I_t=I/I_p$ is a
commutative profinite group, called the \emph{tame inertia group.}
We have a canonical isomorphism
\begin{equation}\label{theta1}
\theta: I_t\xra{\sim}
\varprojlim_{(d,p)=1}\mu_d,
\end{equation}
 where  the projective
system is taken over positive integers prime to $p$, $\mu_d$ is the
group of $d$-th roots of unity in $k$, and the transition maps
$\mu_{m}\ra \mu_d$ are given by  $\zeta\mapsto \zeta^{m/d}$,
whenever $d$ divides $m$. We denote by $\theta_d:I_t\ra \mu_d$ the
projection induced by \eqref{theta1}. Let $q$ be a power of $p$,
$\F_q$ be the finite subfield of $k$ with $q$ elements. Then
$\mu_{q-1}=\F_q^\times$, and we can write $\theta_{q-1}:I_t\ra
\F_q^\times$. The character $\theta_d$ is characterized by the
following property.

\begin{prop}[\cite{Se} Prop.7]\label{prop-V-alpha}
Let $a,d$  be relatively prime positive integers with $d$ prime to $p$. Then the natural action of $I_p$ on the $k$-vector
space $V_{a/d}$ \eqref{defn-V_alpha} is trivial, and  the induced
action of $I_t$ on $V_{a/d}$ is given by the character
$(\theta_d)^a:I_t\ra \mu_d$. In particular, if $q$ is a power of
$p$, the action of $I_t$ on $V_{1/(q-1)}$ is given by the character
$\theta_{q-1}:I_t\ra \F_q^\times$ and any $I$-equivariant
$\F_p$-subspace of $V_{1/(q-1)}$ is an $\F_{q}$-vector space.
\end{prop}

\subsection{}\label{defn-hw-index} Let $G$ be a BT-group over $S$. We define $h(G)$ to be the
valuation of the determinant of  a matrix of $\HW_G$ if $\dim(G^\vee)\geq 1$, and $h(G)=0$ if $\dim(G^\vee)=0$. We call $h(G)$
the \emph{Hasse invariant} of $G$.  

 {(a)} $h(G)$ does not
depend on the choice of the matrix representing $\HW_G$. Indeed, let
$c$ be the rank of $\Lie(G^\vee)$ over $A$,
 $\h\in \rM_{c\times c}(A)$ be a matrix  of $\HW_G$. Any other matrix representing $\HW_G$
can be written in the form $U^{-1}\cdot \h\cdot U^{(p)}$, where
$U\in \GL_{c}(A)$, $U^{-1}$ is the inverse of $U$, and $U^{(p)}$
is the matrix obtained by applying the Frobenius map of $A$ to the
coefficients of $U$.

{(b)} By  \ref{prop-ord}, the generic fiber $G_\eta$ is ordinary if
and only if $h(G)<\infty$; $G$ is ordinary over $T$ if and only
$h(G)=0$.

{(c)} Let $0\ra G'\ra G\ra G''\ra 0$ be a short exact sequence of
BT-groups over $T$, then we have $h(G)=h(G')+h(G'')$. Indeed, the
exact sequence of BT-groups induces a short exact sequence of Lie
algebras (cf. \cite{BBM} 3.3.2)
\[
0\ra\Lie(G''^\vee)\ra\Lie(G^\vee)\ra \Lie(G'^\vee)\ra 0,
\]
from which our assertion follows easily.

\begin{prop}\label{prop-Hasse} 
Let $G$ be a BT-group over $S$. Then we have $h(G)=h(G^\vee)$.
\end{prop}
\begin{proof} The proof is very similar to that of Lemma \ref{lemma-a-number}.   First, we have 
\[h(G)=\leng\bigl(\Lie(G^\vee)/\widetilde{\varphi_G}(\Lie(G^\vee)^{(p)})\bigr),\]
where $\widetilde{\varphi_G}$ is the linearization of $\varphi_G$, and ``$\leng$'' means the length of a finite $A$-module (note that this formulae holds even if $\dim(G^\vee)=0$). By the commutative diagram \eqref{diag-Dieud}, we have 
\[h(G)=\leng\M(G)/(\phi_G(\Lie(G^\vee)^{(p)})+ \omega_G).\]
On the other hand, by applying the functor $\Hom_A(\_,A)$ to the $A$-linear map $\widetilde{\varphi_{G^\vee}}:\Lie(G)^{(p)}\ra \Lie(G)$, we obtain a map $\psi_G:\omega_{G}\ra \omega_G^{(p)}$. If $U$ is a matrix of $\widetilde{\varphi_{G^\vee}}$, then the transpose of $U$, denoted by $U^t$, is a matrix of $\psi_G$. So we have 
\[h(G^\vee)=\tv(\det(U))=\tv(\det(U^t))=\leng\bigl(\omega_G^{(p)}/\psi_G(\omega_G)\bigr).\]
By diagram \ref{diag-Dieud}, we get
\[h(G^\vee)=\leng \M(G)/(\phi_G(\Lie(G^\vee)^{(p)})+ \omega_G)=h(G).\]

\end{proof}

\subsection{} Let $G$ be a BT-group over $S$, $c=\dim(G^\vee)$.
We put
\begin{equation}\label{Tate-mod}\rT_p(G)=\myprojlim G(n)(\Kb)\end{equation}
  the Tate
module of $G$, where $G(n)$ is the kernel of $p^n:G\ra G$.
 It is a free $\Z_p$-module of rank $\leq c$, and
  the equality holds if and only if
 the generic fiber $G_\eta$ is ordinary. The Galois group $I$ acts continuously on
 $\rT_p(G)$.
 We are interested in the image of the monodromy representation
 \begin{equation}\label{rep-trait}\rho:I=\Gal(\Ks/K)\ra
 \Aut_{\Z_p}(\rT_p(G)).\end{equation}
We denote by  \begin{equation}\label{rho-0}\rhob: I=\Gal(\Ks/K)\ra
\Aut_{\F_p}\bigl(G(1)(\Kb)\bigr)\end{equation}  its reduction mod
$p$.

\begin{thm}[Reformulation of Igusa's theorem]\label{thm-Igusa} Let $G$ be
a connected BT-group over  $S$ of height $2$ and dimension $1$. Then $G$ is versal \eqref{versal} if and only if 
$h(G)=1$; moreover, if this condition is satisfied,  the monodromy
representation $\rho:I\ra
\Aut_{\Z_p}(\rT_p(G))\simeq \Z_p^\times $
 is surjective.
\end{thm}

\begin{proof}
  Since  $\Lie(G^\vee)$ is an $\cO_S$-module free of rank 1,  the
  condition that $h(G)=1$ is equivalent to that any matrix of
  $\varphi_G$  is represented  by a uniformizer of $A$.  Hence the first part of this theorem  follows from Proposition \ref{prop-HW-versal}(ii).

We follow \cite[Thm 4.3]{Ka} to prove the surjectivity of $\rho$ under the assumption that $h(G)=1$.
For each integer $n\geq 1$, let
$$\rho_n:I\ra \Aut_{\Z/p^n\Z}(G(n)(\Kb))\simeq
(\Z/p^n\Z)^\times$$ be the reduction mod $p^n$ of $\rho$, $K_n$ be
the subfield of $\Ks$ fixed by the kernel of $\rho_n$. Then $\rho_n$
induces an injective homomorphism $\Gal(K_n/K)\ra
(\Z/p^n\Z)^\times$. By taking projective limits, we are  reduced to proving
the surjectivity of $\rho_n$ for every $n\geq1$. It suffices to
verify that $$|\im(\rho_n)|=[K_n:K]\geq p^{n-1}(p-1)$$ (then the
equality holds automatically).

We regard $G$ as a formal group over $S$. Then by \cite[3.6]{Ka},
there exists a parameter $X$ of the formal group $G$ normalized by
the condition that  $[\xi](X)=\xi(X)$ for all $(p-1)$-th root of unity
$\xi\in \Z_p$. For such a parameter, we have
\[[p](X)=a_1X^{p}+\alpha X^{p^2}+\sum_{m\geq 2}c_mX^{p(1+m(p-1))}\in A[[X]],\]
where we have $\tv(a_1)=h(G)=1$ by \cite[3.6.1 and 3.6.5]{Ka}, and
$\tv(\alpha)=0$, as $G$ is of height $2$. For each integer $i\geq
0$, we put
\[V^{(p^i)}(X)=a_1^{p^i}X+\alpha^{p^i} X^p+\sum_{m\geq 2}c^{p^i}_mX^{1+m(p-1)}\in A[[X]];\]
then we have $[p^n](X)=V^{(p^{n-1})}\circ V^{(p^{n-2})}\circ\cdots
\circ V(X^{p^n})$. Hence each point of $G(n)(\Kb)$ is given by a
sequence $y_1,\cdots, y_n\in \Ks$ (or simply an element $y_n\in
\Ks$) satisfying the equations
\[\begin{cases}V(y_1)=a_1y_1+\alpha y_1^p+\cdots=0;\\
V^{(p)}(y_2)=a_1^py_2+\alpha^p y_2^p+\cdots=y_1;\\
\vdots\\
V^{(p^{n-1})}(y_n)=a_1^{p^{n-1}}y_{n}+\alpha^{p^{n-1}}y_{n}^{p}+\cdots=y_{n-1}.\end{cases}\]
Let $y_n\in \Ks$ be such that $y_1\neq0$. By considering the Newton
polygons of the equations above, we verify that
$$\tv(y_i)=\frac{1}{p^{i-1}(p-1)}\quad \quad \text{for }1\leq i\leq n.$$ In particular, the
ramification index $e(K_n/K)$ is at least $p^{n-1}(p-1)$. By the
definition of $K_n$, the Galois group $\Gal(\Ks/K_n)$ must fix
$y_n\in \Ks$, \ie $K_n$ is an extension of $K(y_n)$. Therefore, we
have $[K_n:K]\geq [K(y_n):K]\geq e(K(y_n)/K)\geq p^{n-1}(p-1)$.
\end{proof}

\begin{prop}\label{prop-mono-trait} Let $G$ be a HW-cyclic BT-group
 over $S$ of height $c+d$ and dimension $d$ such that $G\otimes K$ is ordinary, 
$$\h=\begin{pmatrix}0& 0 &\cdots&0 &-a_1\\
 1&0&\cdots &0 &-a_2\\
 0&1 &\cdots &0 &-a_3\\
 \vdots&&\ddots&& \vdots\\
0&0&\cdots &1&-a_{c} \end{pmatrix}$$ be a matrix of $\HW_G$. Put 
 $q=p^{c}$, $a_{c+1}=1$, and $P(X)=\sum_{i=0}^c a_{i+1}X^{p^i}\in A[X]$.

 \emph{(i)} Assume that $G$ is connected and the Hasse invariant $h(G)=1$. Then the representation $\rhob$ \eqref{rho-0} is
 tame,  $G(1)(\Kb)$
 is endowed with the structure of an $\F_{q}$-vector space of
 dimension $1$, and the induced action of $I_t$
  is given by the character $\theta_{q-1}:I_t\ra
  \F_{q}^\times$.

  \emph{(ii)} Assume that $c>1$, $\tv(a_{i})\geq 2$ for $1\leq i\leq c-1$ and $\tv(a_{c})=1$.
  Then the order of  $\im (\rhob)$ is divisible by $p^{c-1}(p-1)$.

\emph{(iii)} Put $i_0=\min_{0\leq i\leq c}\{i;\tv(a_{i+1})=0\}$. Assume that there exists $\alpha\in k$ such that
 $\tv(P(\alpha))=1$. Then we have $i_0\leq c-1$ and the order of
 $\im(\rhob)$ is divisible by $p^{i_0}$.
\end{prop}

\begin{proof}
  Since $G$ is generically ordinary, we have $a_1\neq 0$ by \ref{prop-ord}(d). Hence $P(X)\in K[X]$ is a separable polynomial. By  \ref{lemma-HW-V},
 $G(1)(\Kb)\simeq (\Ker V_G)(\Ks)$ is identified with the additive group consisting of the roots of $P(X)$ in $\Ks$.

(i) By definition of the Hasse invariant, we have  $\tv(a_1)=h(G)=1$. By \ref{lemma-HW-V}(ii), the assumption that $G$ is connected 
is equivalent to saying $\tv(a_i)\geq 1$ for $1\leq i\leq
c$. From the Newton polygon of $P(X)$, we deduce that all the
non-zero roots of $P(X)$ in $\Ks$ have the same valuation $1/(q-1)$.
We denote by
 \[
 \psi: G(1)(\Kb)\ra V_{1/(q-1)}
 \] the  map
which sends each root $x\in \Ks$ of $P(X)$
 to the class of $x$ in $V_{1/(q-1)}=\m_{1/(q-1)}/\m^+_{1/(q-1)}$ \eqref{defn-V_alpha}. We
remark that $G(1)(\Kb)$ is an $\F_p$-vector space of dimension
$c$. Hence $G(1)(\Kb)$ is automatically of dimension  $1$ over
$\F_q$ once we know it is an $\F_q$-vector space. By
\ref{prop-V-alpha}, it suffices to show that $\psi$ is an injective
$I$-equivariant homomorphism of groups.  By
\ref{lemma-HW-V}(i), $\psi$ is obviously an $I$-equivariant
homomorphism of groups. Let $x_0$ be a root of $P(X)$, and put
$Q(y)=P(x_0y)$. Then the polynomial $Q(y)$ has the form
$Q(y)=x_0^{q}Q_1(y)$, where
$$Q_1(y)=y^{q}+b_{c}y^{p^{c-1}}+\cdots +b_2 y^{p}+b_1y$$
with $b_i=a_i/x_0^{(q-p^{i-1})}\in \Ks$. We have $\tv(b_i)> 0$ for
$2\leq i\leq c$ and $\tv(b_1)=0$. Let $\overline{b}_1$ be the
class of $b_1$ in the residue field $k=\m_0/\m_0^+$. Then the images
of the roots of $P(X)$ in $V_{1/(q-1)}$ are $x_0
\overline{b}_1^{1/(q-1)}\zeta,$ where $\zeta$ runs over the finite
field $\F_q$. Therefore, $\psi$ is injective.

(ii) By computing the slopes of the Newton polygon of $P(X)$, we see
that $P(X)$ has $p^{c-1}(p-1)$ roots of valuation
$1/(p^{c}-p^{c-1})$. Let $L$ be the sub-extension of $\Ks$
obtained  by adding to $K$  all the roots of $P(x)$. Then the
ramification index $e(L/K)$ is divisible by
$p^{c-1}(p-1)$. Let $\widetilde{L}$ be the
sub-extension of $\Ks$ fixed by the kernel of $\rhob$ \eqref{rho-0}.
The Galois group $\Gal(\Ks/\widetilde{L})$ fixes the roots of $P(x)$
by definition. Hence we have $L\subset \widetilde{L}$, and $|\im
(\rhob)|=[\widetilde{L}:K]$ is divisible by $[L:K]$;  in particular,
it is divisible by $p^{c-1}(p-1)$.

(iii) Note that the relation $ i_0\leq c-1$ is equivalent to saying that $G$ is not connected by \ref{lemma-HW-V}(ii). Assume conversely 
$i_0=c$, \ie  $G$ is connected. Then we would have
$${P}(X)\equiv X^{q} \mod (\pi A[X]).$$ But
$\tv(P(\alpha))=1$ implies that $\alpha^{p^{c}}\in \pi A$, \ie
$\alpha=0$; hence we would have $P(\alpha)=0$, which contradicts the
condition $\tv(P(\alpha))=1$.

We put $Q(X)=P(X+\alpha)=P(X)+P(\alpha)$. As $\tv(P(\alpha))=1$,
then $(0,1)$ and $(p^{i_0}, 0)$ are the first two break points of
the Newton polygon of $Q(X)$. Hence there exists $p^{i_0}$ roots of
$Q(X)$ of valuation $1/p^{i_0}$. Let $L$ be the subextension of $K$
in $\Ks$ generated by the roots of $P(X)$. The ramification index
$e(L/K)$ is divisible by $p^{i_0}$. As in the proof of (ii), if $\widetilde{L}$ is the
subextension of $\Ks$ fixed by the kernel of $\rhob$, then it is an
extension of $L$. Therefore, we have $|\im(\rhob)|=[\widetilde{L}:K]$ is divisible by
$[L:K]$, and in particular, divisible by $p^{i_0}$.
\end{proof}

\subsection{} Let $G$ be a BT-group over $S$ with connected part $G^\circ$, and
 \'etale part $G^\et$ of height $r$.
  We have a canonical exact
sequence of $I$-modules
\begin{equation}\label{exseq-Tate-mod-p}
0\ra G^\circ(1)(\Kb)\ra G(1)(\Kb)\ra G^\et(1)(\Kb)\ra 0
\end{equation} giving  rise to a class  $\cb\in
\Ext^1_{\F_p[I]}(G^\et(1)(\Kb), G^\circ(1)(\Kb))$, which vanishes if
and only if \eqref{exseq-Tate-mod-p} splits. Since $I$ acts
trivially on $G^\et(1)(\Kb)$, we have an isomorphism of $I$-modules
$G^\et(1)(\Kb)\simeq\F_p^r$.  Recall that for any $\F_p[I]$-module
$M$, we have a canonical isomorphism (\cite{CL} Chap.VII, \S2)
\begin{equation*}\Ext^1_{\F_p[I]}(\F_p,M)\simeq
H^1(I,M).\end{equation*} Hence we deduce that
\begin{equation}\label{class-cb}\overline{C}\in
\Ext^1_{\F_p[I]}(G^\et(1)(\Kb), G^\circ(1)(\Kb))\simeq H^1(I,
G^\circ(1)(\Kb))^{r}.\end{equation}

\begin{prop}\label{prop-class-coh} Let $G$ be a HW-cyclic BT-group over $S$ such that
$h(G)=1$,  $\overline{\rho}$ \eqref{rho-0} be the representation
of $I$ on $G(1)(\Kb)$. Then the cohomology class $\overline{C}$
 does not vanish if and only if the order of the group
$\im(\rhob)$ is divisible by $p$.
\end{prop}

First, we prove the following result on cohomology of groups.

\begin{lemma}\label{lemma-comm-coh} Let $F$ be a field,
 $\Gamma$ be a commutative  group,
and  $\chi:\Gamma\ra F^\times$ be a non-trivial character of $\Gamma$. We denote
by $F(\chi)$ an $F$-vector space of dimension $1$ endowed with an
action
 of $\Gamma$ given by $\chi$. Then we have $H^1(\Gamma, F(\chi))=0$.
\end{lemma}

\begin{proof}
Let $C$ be a  $1$-cocycle of $\Gamma$ with values in $F(\chi)$. We
prove that $C$ is a $1$-coboundary. For any $g,h\in \Gamma$, we have
\begin{align*}C(gh)=C(g)+{\chi}(g)C(h),\\
C(hg)=C(h)+{\chi}(h)C(g).\end{align*}  Since $\Gamma$ is
commutative, it follows from the relation $C(gh)=C(hg)$ that
\begin{equation}\label{formula-1}({\chi}(g)-1)C(h)=({\chi}(h)-1)C(g).\end{equation} If
${\chi}(g)\neq 1$ and ${\chi}(h)\neq 1$, then
$$\frac{1}{{\chi}(g)-1}C(g)=\frac{1}{{\chi}(h)-1}C(h).$$
Therefore, there exists $x\in \F_q(\chib)$ such that
$C(g)=({\chi}(g)-1)x$ for all $g\in \Gamma$ with ${\chi}(g)\neq 1$.
If ${\chi}(g)=1$, we have also $C(g)=0=({\chi}(g)-1)x$ by
\eqref{formula-1}. This shows that $C$ is a 1-coboundary.
\end{proof}

\begin{proof}[Proof of $\ref{prop-class-coh}$]
By  \ref{lemma-cyclic}(ii) and \ref{defn-hw-index}(c), the connected
part $G^\circ$ of $G$ is HW-cyclic with $h(G^\circ)=h(G)=1$.
Assume that $\rT_p(G^\circ)$ has rank $\ell$ over $\Z_p$, and
$\rT_p(G^\et)$ has rank $r$. Then by \ref{prop-mono-trait}(a),
$G^\circ(1)(\Kb)$ is an $\F_q$-vector space of dimension 1 with
$q=p^\ell$, and the action of $I$ on $G^\circ(1)(\Kb)$ factors through
the character $\chib:I\ra I_t\xra{\theta_{q-1}} \F_q^{\times}$. We
write  $G^\circ(1)(\Kb)=\F_q(\chib)$ for short. If the cohomology class
$\overline{C}$ is zero, then the exact sequence
\eqref{exseq-Tate-mod-p} splits, \ie we  have an isomorphism of
Galois modules $G(1)(\Kb)\simeq \F_q(\chi)\oplus \F_p^r$. It is
clear that the group $\im(\rhob)$ has order $q-1$.

Conversely, if the cohomology class $\cb$ is not zero, we will show
that there exists an element in $\im(\rhob)$ of order $p$. We choose
a basis adapted to the exact sequence \eqref{exseq-Tate-mod-p} such
that the action of $g\in I$ is given by
\begin{equation}\label{formula-action}\rhob(g)=\begin{pmatrix}\chib(g)&\cb(g)\\
0&\mathbf{1}_r\end{pmatrix},\end{equation}
 where $\mathbf{1}_r$ is the unit matrix of
type $(r,r)$ with coefficients in $\F_p$, and the map
$g\mapsto\cb(g)$ gives rise to a 1-cocycle representing the
cohomology class $\cb$. Let $I_1$ be the kernel of $\chib:I\ra
\F_q^\times$,  $\Gamma$ be the quotient $I/I_1$, so $\chib$ induces
an isomorphism $\chib: \Gamma\xra{\sim} \F_q^\times$. We have an
exact sequence
\[0\ra H^1(\Gamma,\F_q(\chib))^r\xra{\mathrm{Inf}}H^1(I,\F_q(\chib))^r\xra{\mathrm{Res}}H^1(I_1,\F_q(\chib))^r,\]
where ``Inf'' and ``Res'' are respectively the inflation and
restriction homomorphisms in group cohomology. Since $H^1(\Gamma,
\F_q(\chib))^r=0$ by \ref{lemma-comm-coh}, the restriction of the
cohomology class $\cb$ to $H^1(I_1,\F_q(\chib))^r$ is non-zero.
Hence there exists $h\in I_1$ such that $\cb(h)\neq 0$. As we have
$\chib(h)=1$, then
\[\rhob(h)^{p}=\begin{pmatrix}\mathbf{1}_\ell&p\cb(h)\\
0&\mathbf{1}_r\end{pmatrix}=\mathbf{1}_{\ell+r}.\] Thus the order of
$\rhob(h)$ is  $p$.
\end{proof}

\begin{cor}\label{cor-non-zero-coh} Let $G$ be a HW-cyclic BT-group over $S$,
$$\h=\begin{pmatrix}0& 0 &\cdots&0 &-a_1\\
 1&0&\cdots &0 &-a_2\\
 0&1&\cdots &0 &-a_3\\
 \vdots&&\ddots&& \vdots\\
0&0&\cdots &1&-a_{c} \end{pmatrix}$$ be
 a matrix
of $\HW_G$, $P(X)=X^{p^{c}}+a_{c}X^{p^{c-1}}+\cdots+a_1X\in
A[X]$. If $h(G)=1$ and if there exists $\alpha\in k\subset A$ such
that $\tv(P(\alpha))=1$, then the cohomology class \eqref{class-cb}
is not zero, \ie the extension
 of  $I$-modules \eqref{exseq-Tate-mod-p} does not split.
\end{cor}

\begin{proof} Since $\tv(a_1)=h(G)=1$, the integer $i_0$ defined in
  \ref{prop-mono-trait}(iii) is at least $1$. Then the corollary follows
  from \ref{prop-mono-trait}(iii) and \ref{prop-class-coh}.
\end{proof}

\section{Lemmas in Group Theory}
In this section, we fix a  prime number $p\geq 2$ and an integer $n\geq 1$.

\subsection{}  Recall that the general linear group $\GL_n(\Z_p)$ admits a natural exhaustive
decreasing filtration by  normal subgroups
\[\GL_n(\Z_p)\supset 1+p\rM_n(\Z_p)\supset \cdots\supset 1+p^m\rM_n(\Z_p)\supset\cdots,\]
where $\rM_n(\Z_p)$ denotes the ring of matrix of type $(n,n)$ with
coefficients in $\Z_p$.  We endow $\GL_n(\Z_p)$  with the topology
for which $(1+p^m\rM_n(\Z_p))_{m\geq 1}$ form a fundamental system
of neighborhoods of $1$. Then $\GL_n(\Z_p)$ is a complete and
separated topological group.

\subsection{}Let  $\fG$ be a profinite group,  $\rho:\fG\ra
\GL_n(\Z_p)$ be a continuous homomorphism of topological groups. By
taking inverse images, we obtain a decreasing filtration $(F^m\fG,
m\in \Z_{\geq 0})$ on $\fG$ by open normal subgroups:
\[F^0\fG=\fG, \quad \text{and }\quad
F^m\fG=\rho^{-1}(1+p^m\rM_n(\Z_p))\;\, \text{for $m\geq 1$}.\]
Furthermore, the homomorphism $\rho$ induces a sequence of injective
homomorphisms of finite groups
\begin{align}\label{phi-gradue}&\rho_0\colon F^0\fG/F^1\fG\lra \GL_n(\F_p)\\
&\rho_m\colon F^m\fG/F^{m+1}\fG\ra \rM_n(\F_p),\quad \text{for
}m\geq 1.\end{align}

\begin{lemma}\label{lemma-gp-1} The homomorphism $\rho$ is surjective if and only if
the following conditions are satisfied:

\emph{(i)} The homomorphism $\rho_0$ is surjective.

\emph{(ii)} For every integer $m\geq 1$, the subgroup $\im(\rho_m)$
of $ \rM_n(\F_p)$ contains  an element  of the form
$$\begin{pmatrix}x& 0&\cdots &0\\0&0&\cdots&0\\
\vdots&\vdots &\ddots&\vdots\\
0&0&\cdots &0\end{pmatrix}$$ with $x\neq 0$; or
equivalently, there exists, for every $m\geq 1$, an element $g_m\in
\fG$ such that $\rho(g_m)$ is of the
form $$\begin{pmatrix}1+p^ma_{1,1} &p^{m+1}a_{1,2} &\cdots &p^{m+1}a_{1,n}\\
p^{m+1}a_{2,1} &1+p^{m+1}a_{2,2} &\cdots & p^{m+1}a_{2,n}\\
\vdots&\vdots&\ddots&\vdots\\
p^{m+1}a_{n,1}& p^{m+1}a_{n,2}&\cdots& 1+p^{m+1}a_{n,n}
\end{pmatrix},$$ where $a_{i,j}\in \Z_p$ for $1\leq i,j\leq n$ and $a_{1,1}$
is not divisible by $p$.
\end{lemma}

\begin{proof}We notice first that $\rho$ is surjective if and only if $\rho_m$ is surjective for every $m\geq 0$, because $\fG$ is complete and $\GL_n(\Z_p)$ is separated \cite[Chap. III \S2 $\mathrm{n}^\circ8$ Cor.2 au Th\'eo. 1]{Bou}. The surjectivity of $\rho_0$ is condition (i). Condition (ii) is clearly necessary. We prove that it implies the surjectivity of $\rho_m$ for all $m\geq 1$, under the assumption of (i). First, we remark that under condition (i), if $A$ lies in $\im(\rho_m)$, then for any $U\in \GL_n(\F_p)$ the conjuagate matrix $U\cdot A\cdot U^{-1}$ lies also in $\im(\rho_m)$. In fact, let $\widetilde{A}$ be a lift of $A$ in $\rM_n(\Z_p)$ and $\widetilde{U}\in \GL_n(\Z_p)$ a lift of $U$. By assumption, there exist $g,h\in \fG$ such that 
\[\rho(g)\equiv 1+p^m\widetilde{A}\mod (1+p^{m+1}M_n(\Z_p))\quad\text{and}\quad \rho(h)\equiv \widetilde{U}\mod(1+p\rM_n(\Z_p)).\]
Therefore, we have $\rho(hgh^{-1})\equiv (1+p^m\widetilde{U}\cdot\widetilde{A}\cdot \widetilde{U}^{-1})\mod (1+p^{m+1}\rM_n(\Z_p))$. Hence $hgh^{-1}\in F^m\fG$ and $\rho_m(hgh^{-1})=U\cdot A \cdot U^{-1}$.

For $1\leq i,j\leq n$, let $E_{i,j}\in \rM_n(\F_p)$ be the matrix
whose $(i,j)$-th entry is $0$ and the other entries are $0$.  The
matrices $E_{i,j}(1\leq i,j\leq n)$ form clearly a basis of $\rM_n(\F_p)$ over $\F_p$. To prove the surjectivity of $\rho_m$, we only need to verify that $E_{i,j}\in \im(\rho_m)$ for $1\leq i,j \leq n$, because $\im(\rho_m)$ is an $\F_p$-subspace of $\rM_n(\F_p)$. By assumption, we have $E_{1,1}\in \im(\rho_m)$. For $2\leq i \leq n$, we put $U_i=E_{1,i}-E_{i,1}+\sum_{j\neq 1,i}E_{j,j}$. Then we have $U_i\in \GL_n(\Z_p)$ and $U_i\cdot E_{1,1}\cdot U_i^{-1}=E_{i,i}\in \im(\rho_m)$. For  $1\leq  i< j\leq n$, we put $U_{i,j}=I+E_{i,j}$ where $I$ is the unit matrix. Then we have $U_{i,j}\cdot E_{i,i}\cdot U^{-1}_{i,j}=E_{i,i}+E_{i,j}\in \im(\rho_m)$, and hence $E_{i,j}\in \im(\rho_m)$. This completes the proof.

\end{proof}

\begin{rem} By using the arguments in \cite[Chap. IV 3.4 Lemma 3]{Se2}, we can prove the following stronger form of  Lemma \ref{lemma-gp-1}: \emph{If $p=2$, condition $\mathrm{(i)}$ and $\mathrm{(ii)}$ for $m=1,2$ are sufficient to guarantee the surjectivity of $\rho$; if $p\geq 3$, then $\mathrm{(i)}$ and $\mathrm{(ii)}$ just for $m=1$ suffice already.}
\end{rem}

A subgroup $C$ of $\GL_n(\F_p)$ is called a \emph{non-split Cartan subgroup}, if the subset  $C\cup \{0\}$ of the matrix algebra $\rM_n(\F_p)$ is a field isomorphic to $\F_{p^n}$; such a group is cyclic of order $p^n-1$. 

\begin{lemma}\label{lemma-gp-2}
Assume that $n\geq 2$. We denote by $H$ the subgroup of $\GL_n(\F_p)$ consisting of all the elements of the form  
$\begin{pmatrix}A&b\\
0&1\end{pmatrix},$ where $A\in \GL_{n-1}(\F_p)$ and $b=\begin{pmatrix}b_1\\ \vdots\\ b_{n-1}
\end{pmatrix}$ with $b_i\in \F_p(1\leq i\leq n-1)$. 
 Let $G$ be a subgroup  of $\GL_n(\F_p)$. Then $G=\GL_n(\F_p)$  if and only if $G$ contains $H$ and a non-split Cartan subgroup of $\GL_n(\F_p)$.
\end{lemma}

\begin{proof} The ``only if'' part is clear. For the ``if'' part, let
  $C$ be a non-split Cartan subgroup contained in $G$. For a finite
  group $\Lambda$, we denote by $|\Lambda|$ its order.  An easy
  computation shows that $|\GL_n(\F_p)|=|H| \cdot |C|$. So we just
  need to prove that $U\cap C=\{1\}$; since then we will have
  $|\GL_n(\F_p)|=|G|$, hence $G=\GL_n(\F_p)$. Let $g\in H\cap C$, and
  $P(T)\in \F_p[T]$ be its characteristic polynomial.  We fix an
  isomorphism $C\simeq \F_{p^n}^\times$, and let $\zeta\in
  \F_{p^n}^\times$ be the element corresponding to $g$. We have
  $P(T)=\prod_{\sigma\in \Gal(\F_{p^n}/\F_p)}(T-\sigma(\zeta))$ in
  $\F_{p^n}[T]$. On the other hand, the fact that $g\in H$  implies that $(T-1)$ divise $P(T)$. Therefore, we get $\zeta=1$, \ie $g=1$.
\end{proof}

\begin{rem} E. Lau point  out the following strengthened version of \ref{lemma-gp-2}: \emph{When $n\geq 3$, a subgroup $G\subset\GL_n(\F_p)$ coincides with $\GL_n(\F_p)$ if and only if $G$ contains a non-split Cartan subgroup and the subgroup $\begin{pmatrix}\GL_{n-1}(\F_p)&0\\
0&1\end{pmatrix}$}. This can be used to simplify the induction process in the proof  of Theorem \ref{thm-one-dim} when $n\geq 3$.
\end{rem}

\section{Proof of Theorem \ref{thm-main} in the One-dimensional Case}

\subsection{} We start with a general remark on the monodromy of
BT-groups. Let $X$ be a scheme, $G$ be an ordinary BT-group over a scheme $X$,
$G^\et$ be its \'etale part \eqref{decom-ord}. If $\etab$ is a
geometric point of $X$, we denote by
\[\rT_p(G,\etab)=\varprojlim_nG(n)(\etab)=\varprojlim_n G^\et(n)(\etab)\]
the Tate module of $G$ at $\etab$, and by $\rho(G)$ the monodromy
representation of $\pi_1(X,\etab)$ on $\rT_p(G,\etab)$. Let $f:Y\ra
X$ be a morphism of schemes, $\xib$ be a geometric point of $Y$,
$G_Y=G\times_X Y$. Then by the functoriality,  we have a commutative
diagram
\begin{equation}\label{funct-mono}\xymatrix{\pi_1(Y,\xi)\ar[r]^{\pi_1(f)}\ar[d]_{\rho(G_Y)}&\pi_1(X,f(\xib))\ar[d]^{\rho(G)}\\
\Aut_{\Z_p}(\rT_p(G_Y,\xib))\ar@{=}[r]&\Aut_{\Z_p}(\rT_p(G,f(\xib)))}\end{equation}
In particular, the monodromy of $G_Y$ is a subgroup of the monodromy
of $G$. In the sequel,  diagram \eqref{funct-mono} will be
refereed as the \emph{functoriality of monodromy} for the BT-group
$G$ and the morphism  $f$.

\subsection{} Let $k$ be an algebraically closed field of
characteristic $p>0$, $G$ be the unique connected  BT-group over $k$
of dimension $1$ and height $n+1\geq 2$ \eqref{HW-exem}. We denote by
$\bS$ the algebraic local moduli of $G$ in characteristic $p$, by
$\bG$ the universal deformation of $G$ over $\bS$, and by $\bU$ the
ordinary locus of $\bG$ over $\bS$ \eqref{defn-moduli}. Recall that $\bS$ is affine of
ring $R\simeq k[[t_1,\cdots, t_n]]$ \eqref{cor-alg-univ}, and that $G$
and $\bG$ are HW-cyclic (cf. \ref{lemma-cyclic}(i) and \ref{HW-exem}). Let $\etab$ be a geometric point of $\bU$ over its generic point. We  put  
$$\rT_p(\bG,\etab)=\varprojlim_{m\in \Z_{\geq 1}}\bG(m)(\etab)$$ to be the Tate module of $\bG$ at the point $\etab$. This is a free $\Z_p$-module of rank $n$. We have the monodromy representation 
\[\rho_n:\pi_1(\bU,\etab)\ra \Aut_{\Z_p}(\rT_p(\bG,\etab))\simeq \GL_n(\Z_p).\]
The following  is the one-dimensional case of Theorem \ref{thm-main}.

\begin{thm}\label{thm-one-dim}
Under the above assumptions, the homomorphism $\rho_n$ is surjective for $n\geq 1$.
\end{thm}

\subsection{}\label{nota-one-dim}  First, we assume $n\geq 2$.  By Proposition \ref{prop-HW-versal}(ii), we may assume that  
\begin{equation}\label{HW-one-dim}
\h=\begin{pmatrix}0& 0 &\cdots&0 &-t_1\\
 1&0&\cdots &0 &-t_2\\
 0&1&\cdots &0 &-t_3\\
 \vdots&&\ddots&& \vdots\\
0&0&\cdots &1&-t_{n}\end{pmatrix}
\end{equation}is a matrix of the Hasse-Witt map $\varphi_{\bG}$.
Let  $\fp$ be the prime ideal of $R$ generated by
$t_1,\cdots,t_{n-1}$,  $K_0\simeq k((t_n))$  be the fraction field of
$R/\fp$, $R'=\widehat{R}_{\fp}$ be the completion of the localization
of $R$ at $\fp$, and $\cG_{R'}=\bG\otimes_{R}R'$. Since the natural
map $R\ra R'$ is injective, for any $a\in R$, we will denote also by
$a$ its image in $R'$. Since the Hasse-Witt map commutes with base
change, the image of $\h$ in $\rM_{n\times n}(R')$, denoted also by
$\h$, is a matrix of $\varphi_{\cG_{R'}}$. Applying  \ref{lemma-HW-V}(ii) to the closed point of $\Spec(R')$, we see that the \'etale part of $\cG_{R'}$ has height $1$ and its connected part $\cG^\circ_{R'}$ has height $n$. We have an exact sequence of BT-groups over $R'$
\begin{equation}\label{seq-1}
0\ra \cG^{\circ}_{R'}\ra \cG_{R'}\ra \cG_{R'}^\et\ra 0.
\end{equation}
We fix  an imbedding $i:K_0\ra \Kb_0$ of $K_0$ 
into an algebraically closed field. Put
$\cG^{*}_{\Kb_0}=\cG^{*}_{R'}\otimes \Kb_0$  for
$*=\emptyset,\et,\circ$.  We have $\cG^\et_{\Kb_0}\simeq \Q_p/\Z_p$,
and $\cG^\circ_{\Kb_0}$ is the unique connected one-dimensional
BT-group over $\Kb_0$ of height $n$ (cf. \ref{HW-exem}).  We put
$\Rb=\Kb_0[[x_1,\cdots, x_{n-1}]]$, and 
\begin{equation}\label{defn-Sigma}
\Sigma=\{\text{ring homomorphisms }\sigma:R'\ra \Rb \text{ lifting }R'\ra K_0\xra{i} \Kb_0\}
\end{equation} 
Let $\sigma\in \Sigma$. We deduce from \eqref{seq-1} by base change an exact sequence of BT-groups over $\Rb$
\begin{equation}\label{seq-BT-R'}
0\ra \cG^\circ_{\Rb,\sigma}\ra \cG_{\Rb,\sigma}\ra \cG^\et_{\Rb,\sigma}\ra 0,
\end{equation}
where we have put $\cG^*_{\Rb,\sigma}=\cG^*_{R'}\otimes_{\sigma}\Rb$ for $*=\circ, \emptyset, \et$.
 Due to the henselian property of $\Rb$, the isomorphism $\cG^\et_{\Kb_0}\simeq \Q_p/\Z_p$ lifts uniquely to an isomorphism $\cG^\et_{\Rb,\sigma}\simeq \Q_p/\Z_p$ .
 Assume that $\cG^\circ_{\Rb,\sigma}$ is generically ordinary over $\Sb=\Spec(\Rb)$. Let $\widetilde{U}_\sigma'\subset \Sb$ be its ordinary locus, and $\xb$ be a geometric point over the generic point of $\widetilde{U}_\sigma'$. The exact sequence \eqref{seq-BT-R'} induces an exact sequence of Tate modules 
\begin{equation}\label{filt-Tate-R'}
0\ra \rT_p(\cG^\circ_{\Rb,\sigma},\xb)\ra \rT_p(\cG_{\Rb,\sigma},\xb)\ra \rT_p(\cG^{\et}_{\Rb,\sigma},\xb)\ra0
\end{equation}
compatible with the actions of $\pi_1(\widetilde{U}_\sigma',\xb)$. Since we have $\rT_p(\cG^\et_{\Rb,\sigma},\xb)\simeq\rT_p(\Q_p/\Z_p,\xb)= \Z_p$, this determines a cohomology 
class
\begin{equation}\label{class-one}C_\sigma\in \Ext^1_{\Z_p[\pi_1(\widetilde{U}_\sigma',\xb)]}(\Z_p,\rT_p(\cG^\circ_{\Rb,\sigma},\xb))\simeq H^1(\pi_1(\widetilde{U}_\sigma',\xb),\rT_p(\cG^\circ_{\Rb,\sigma},\xb)).\end{equation}
We consider also the ``mod-$p$ version'' of \eqref{filt-Tate-R'}
\[0\ra \cG^\circ_{\Rb,\sigma}(1)(\xb)\ra \cG_{\Rb,\sigma}(1)(\xb)\ra \F_p\ra 0,\]
which determines a cohomology class 
\begin{equation}\label{class-mod-p-one}\overline{C}_\sigma\in \Ext^1_{\F_p[\pi_1(\widetilde{U}_\sigma',\xb)]}(\F_p,\cG^\circ_{\Rb,\sigma}(1)(\xb))\simeq H^1(\pi_1(\widetilde{U}_\sigma',\xb),\cG^\circ_{\Rb,\sigma}(1)(\xb)).\end{equation}
It is clear that $\overline{C}_\sigma$ is the image of $C_\sigma$ by the canonical reduction map \[H^1(\pi_1(\widetilde{U}_\sigma',\xb),\rT_p(\cG^\circ_{\Rb,\sigma},\xb))\ra H^1(\pi_1(\widetilde{U}_\sigma',\xb),\cG^\circ_{\Rb,\sigma}(1)(\xb)).\]

\begin{lemma}\label{lemma-key} Under the above assumptions, there exist $\sigma_1,\sigma_2\in \Sigma$ satisfying the following properties:

\emph{(i)} We have $\cG^\circ_{\Rb,\sigma_1}=\cG^\circ_{\Rb,\sigma_2}$, and it is the universal deformation of $\cG^\circ_{\Kb_0}$.

\emph{(ii)} We have $C_{\sigma_1}=0$ and $\cb_{\sigma_2}\neq0$.
\end{lemma}

Before proving  this lemma, we prove first  Theorem \ref{thm-one-dim}.

\begin{proof}[Proof of \ref{thm-one-dim}] First, we notice that the monodromy of a BT-group is independent of the base point. So we can change  $\etab$ to any other geometric point of $\bU$ when discussing the monodromy of $\bG$.  We make an induction on the codimension $n=\dim(G^\vee)$. The case of $n=1$ is proved in Theorem \ref{thm-Igusa}. Assume that $n\geq 2$ and the theorem is proved for $n-1$.  We denote by 
$$\rhob_n:\pi_1(\bU,\etab)\ra \Aut_{\F_p}(\bG(1)(\etab))\simeq
  \GL_n(\F_p)$$ the reduction of $\rho_n$ modulo by $p$.  By Lemma
  \ref{lemma-gp-1} and \ref{lemma-gp-2}, to prove the surjectivity of
  $\rho_n$, we only need to verify the following conditions:

(a) $\im(\rhob_n)$ contains a non-split Cartan subgroup of $\GL_n(\F_p)$; 

(b) $\im(\rho_n)$ contains the subgroup $H\subset\GL_n(\Z_p)$ consisting of  all the elements of the form $\begin{pmatrix}B&b\\0&1\end{pmatrix}\in \GL_n(\Z_p)$,
with $B\in \GL_{n-1}(\Z_p)$ and $b=\rM_{n-1\times 1}(\Z_p)$;

For condition (a), let $A=k[[\pi]]$, $T=\Spec(A)$, $\xi$ be its generic point, $\xib$ be a geometric point over $\xi$, and  $I=\Gal(\xib/\xi)$ be the absolute Galois group over $\xi$. We keep the notations of \ref{nota-one-dim}. Let $f^*:R\ra A$ be the homomorphism of $k$-algebras such that $f^*(t_1)=\pi$ and $f^*(t_i)=0$ for $2\leq i\leq n$. We denote by $f:T\ra \bS$ the corresponding morphism of schemes, and put $G_T=\bG\times_{\bS}T$. By the functoriality of Hasse-Witt maps, 
\[\h_T=\begin{pmatrix}0&0&\cdots&0 &-\pi\\
1&0&\cdots&0&0\\
\vdots&&\ddots&&\vdots\\
0&0&\cdots &1&0\end{pmatrix}\]
is a matrix of $\varphi_{G_T}$. By definition \ref{defn-hw-index}, the Hasse invariant  of $G_T$ is $h(G)=1$. Hence $G_T$ is generically ordinary; so $f(\xi)\in \bU$. Let 
\[\rhob_T:I=\Gal(\xib/\xi)\ra \Aut_{\F_p}(G_T(1)(\xib))\] be the mod-$p$ monodromy representation attached to $G_T$. Proposition \ref{prop-mono-trait}(i) implies that $\im(\rhob_T)$ is a non-split Cartan subgroup of $\GL_n(\F_p)$. On the other hand, by  the  functoriality of monodromy, we get $\im(\rhob_T)\subset \im(\rhob_n)$. This verifies condition (a). 

To check condition (b), we consider the constructions in \ref{nota-one-dim}. Let  $S'=\Spec(R')$, $f:S'\ra \bS$ be the morphism of schemes corresponding to the natural ring homomorphism  $R\ra R'$, $U'$ be the ordinary locus of $\cG_{R'}$, and $\xib$  be a geometric point of $U'$.    From \eqref{seq-1}, we deduce an exact sequence of Tate modules
\begin{equation}\label{seq-Tate-1}0\ra \rT_p(\cG_{R'}^\circ,\xib)\ra \rT_p(\cG_{R'},\xib)\ra \rT_p(\cG_{R'}^\et,\xib)\ra0.\end{equation}
Let $\rho_{\cG'}:\pi_1(U',\xib)\ra
\Aut_{\Z_p}(\rT_p(\cG_{R'},\xib))\simeq \GL_n(\Z_p)$ be the  monodromy
represention of $\cG_{R'}$. Under any basis of $\rT_p(\cG_{R'},\xib)$
adapted to \eqref{seq-Tate-1},  the action of $\pi_1(U',\xib)$ on $\rT_p(\cG_{R'},\xib)$ is given by 
\[\rho_{\cG_{R'}}\colon g\in \pi_1(U',\xib)\mapsto\begin{pmatrix}\rho_{\cG^\circ_{R'}}(g)&*\\
0&\rho_{\cG^{\et}_{R'}}(g),\end{pmatrix}\]
where $g\mapsto\rho_{\cG^\circ_{R'}}(g)\in \GL_{n-1}(\Z_p)$ (\resp
$g\mapsto\rho_{\cG^\et_{R'}}(g)\in \Z_p^\times$) gives the action of
$\pi_1(U',\xib)$ on $\rT_p(\cG^{\circ}_{R'},\xib)$ (\resp on
$\rT_p(\cG^{\et}_{R'},\xib)$). Note that  $f(U')\subset \bU$. So by  the functoriality of monodromy, we get $\im(\rho_{\cG'})\subset \im(\rho_n)$. To complete the proof of Theorem \ref{thm-one-dim}, it suffices to check   condition  (b)
 with $\rho_{n}$ replaced by $\rho_{\cG_{R'}}$  under the induction
hypothesis that \ref{thm-one-dim} is valide for $n-1$.  Let $\sigma_1, \sigma_2:R'\ra \Rb$ be the homomorphisms  given by \ref{lemma-key}. For $i=1,2$, we denote by  $f_{i}:\Sb=\Spec(\Rb)\ra S'=\Spec(R')$ the morphism of schemes corresponding to $\sigma_i$, and  put $\cG_i=\cG_{\Rb,\sigma_i}=\cG_{R'}\otimes_{\sigma_i}\Rb$ to simply the notations. By condition \ref{lemma-key}(i),  we can denote by $\cG^\circ$ the common connected component of $\cG_{1}$ and $\cG_{2}$. Let $\Ub\subset \Sb$ be the ordinary locus of $\cG^\circ$. Then we have $f_i(\Ub)\subset U'$ for $i=1,2$. Let $\xb$ be a geometric point over the generic point of $\Ub$. We have an exact sequence of Tate modules
\begin{equation}\label{seq-proof-one}
0\ra \rT_p(\cG^\circ,\xb)\ra \rT_p(\cG_i,\xb)\ra \rT_p(\Q_p/\Z_p,\xb)\ra 0
\end{equation}
compatible with the actions of $\pi_1(\Ub,\xb)$. We denote by 
$$\rho_{\cG_i}:\pi_1(\Ub,\xb)\ra \Aut_{\Z_p}(\rT_p(\cG_i,\xb))\simeq \GL_n(\Z_p)$$ the monodromy representation of $\cG_i$. In a basis adapted to   \eqref{seq-proof-one}, the action of $\pi_1(\Ub,\xb)$ on $\rT_p(\cG_i,\xb)$ is given by 
\[\rho_{\cG_i}: g\mapsto \begin{pmatrix}\rho_{\cG^\circ}(g)&C_{\sigma_i}(g)\\
0&1\end{pmatrix},\]
where $\rho_{\cG^\circ}:\pi_1(\Ub,\xb)\ra \GL_{n-1}(\Z_p)$ is the monodromy representation of $\cG^\circ$, and the cohomology class in $H^1(\pi_1(\Ub,\xb),\rT_p(\cG^\circ))$ given by  $g\mapsto C_{\sigma_i}(g)$ is nothing but the class defined in \eqref{class-one}.  By \ref{lemma-key}(i) and the induction hypothesis, $\rho_{\cG^\circ}$ is surjective. Since  the cohomology class $C_{\sigma_1}=0$ by \ref{lemma-key}(ii), we may assume $C_{\sigma_1}(g)=0$ for all $g\in \pi_1(U',\xb)$. Therefore $\im(\rho_{\cG_1})$ contains all the matrix of the  form $\begin{pmatrix}B&0\\
0&1\end{pmatrix}$ with $B\in \GL_{n-1}(\Z_p)$. By the functoriality of monodromy, $\im(\rho_{\cG_{R'}})$ contains $\im(\rho_{\cG_1})$. Hence we have  
\begin{equation}\label{mono-im-1}
\begin{pmatrix}\GL_{n-1}(\Z_p)&0\\
0&1\end{pmatrix}\subset \im(\rho_{\cG_1})\subset\im(\rho_{\cG_{R'}}).
\end{equation} On the other hand, since the cohomology class $\cb_{\sigma_2}\neq 0$, there exists a $g\in \pi_1(\Ub,\xb)$ such that $b_2=\cb_{\sigma_2}(g)\neq0$. Hence the matrix $\rho_{\cG_2}(g)$ has the form
$\begin{pmatrix}B_2&b_2\\
0&1\end{pmatrix}$ such that $B_2\in \GL_{n-1}(\Z_p)$ and the image of $b_2\in\rM_{1\times n-1}(\Z_p)$ in $\rM_{1\times n-1}(\F_p)$ is non-zero.  By the functoriality of monodromy, we have $\im(\rho_{\cG_2})\subset\im(\rho_{\cG_{R'}})$; in particular, we have $\begin{pmatrix}B_2&b_2\\
0&1\end{pmatrix}\in \im(\rho_{\cG_{R'}})$.  In view of  \eqref{mono-im-1}, we get
\[\begin{pmatrix}\GL_{n-1}(\Z_p)&0\\
0&1\end{pmatrix}
\begin{pmatrix}B_2&b_2\\
0&1\end{pmatrix}
\begin{pmatrix}\GL_{n-1}(\Z_p)&0\\
0&1\end{pmatrix}
\subset \im(\rho_{\cG_{R'}}).\]
 But  the subset of $\GL_n(\Z_p)$ on the left hand side is just the subgroup $H$ described in condition (b). Therefore, condition (b) is verified for $\rho_{\cG_{R'}}$, and the proof of \ref{thm-one-dim} is complete. 
\end{proof}

The rest of this section is dedicated to the proof of Lemma \ref{lemma-key}.

\begin{lemma}\label{sublemma-1}
Let $k$ be an algebraically closed field of characteristic $p>0$, $A$ be a noetherian henselian local $k$-algebra with residue field $k$, $G$ be a BT-group over $A$,  and $G^\et$ be its \'etale part. Put 
\[\Lie(G^\vee)^{\varphi=1}=\{x\in \Lie(G^\vee)\;\text{such that}\;\, \varphi_G(x)=x\}.\]Then $\Lie(G^\vee)^{\varphi=1}$ is an $\F_p$-vector space of dimension equal to the rank of $\Lie(G^{\et\vee})$, and the $A$-submodule $\Lie(G^{\et\vee})$ of $\Lie(G^\vee)$ is generated by $\Lie(G^\vee)^{\varphi=1}$.
 \end{lemma}

\begin{proof} Let $r$ be the rank of $\Lie(G^{\et\vee})$, $G^\circ$ be the connected part of $G$, and $s$ be the height of $\Lie(G^{\circ\vee})$. We have an exact sequence of  $A$-modules
\[0\ra \Lie(G^{\et\vee})\ra \Lie(G^\vee)\ra \Lie(G^{\circ\vee})\ra 0,\]
compatible with Hasse-Witt maps. We choose a basis of $\Lie(G^\vee)$ adapted to this exact sequence, so that $\varphi_{G}$ is expressed by a matrix of the form $\begin{pmatrix}U&W\\
0&V\end{pmatrix}$ with $U\in \rM_{r\times r}(A)$, $V\in \rM_{s\times s}(A)$, and $W\in \rM_{r\times s}(A)$. An element  of $\Lie(G^\vee)^{\varphi=1}$ is given by a vector $\begin{pmatrix}x\\y\end{pmatrix}$, where $x=\begin{pmatrix}x_1\\ \vdots \\ x_{r}\end{pmatrix}$ and $y=\begin{pmatrix}y_1\\ \vdots \\ y_{s}\end{pmatrix}$ with $x_i, y_j\in A$, satisfying
\begin{equation}\label{formula-sol}\begin{pmatrix}U&W\\
0&V\end{pmatrix}\cdot \begin{pmatrix}x^{(p)}\\ y^{(p)}\end{pmatrix}=\begin{pmatrix}x\\y\end{pmatrix}\quad \Leftrightarrow \quad\begin{cases}U\cdot x^{(p)}+W\cdot y^{(p)}=x\\
V\cdot y^{(p)}=y.\end{cases}
\end{equation} 
where $x^{(p)}$ (\resp $y^{(p)}$) is  the vector obtained by applying $a\mapsto a^p$ to each $x_i(1\leq i\leq r)$ (\resp $y_j(1\leq j\leq s)$). By \ref{cor-nilp-HW}, the Hasse-Witt map of the special fiber of $G^\circ$ is nilpotent. So there exists an integer $N\geq 1$ such that $\varphi_{G^\circ}^N(\Lie(G^{\circ\vee}))\subset\m_{A}\cdot \Lie(G^{\circ\vee})$, \ie we have $V\cdot V^{(p)} \cdots  V^{(p^{N-1})}\equiv 0\quad(\mathrm{mod}\; \m_A)$. From the equation $V\cdot y^{(p)}=y$, we deduce that 
\[y=V\cdot V^{(p)}\cdots V^{(p^{N-1})}\cdot y^{(p^N)}\equiv 0 \quad(\mathrm{mod}\; \m_A).\]
But this implies that $y^{(p^N)}\equiv 0\quad(\mathrm{mod}\; \m_A^{p^N})$. Hence we get $y=V\cdot y^{(p)}\equiv 0\quad(\mathrm{mod}\; \m_A^{p^N+1})$. Repeting this argument, we get finally $y\equiv 0\quad(\mathrm{mod}\; \m_A^\ell)$ for all integers $\ell\geq 1$, so $y=0$. This implies that $\Lie(G^\vee)^{\varphi=1}\subset \Lie(G^{\et\vee})$, and the equation \eqref{formula-sol} is simplified as $U\cdot x^{(p)}=x$. Since the linearization of $\varphi_{G^\et}$ is bijective  by \ref{prop-ord}, we have $U\in\GL_r(A)$. Let $\overline{U}$ be the image of $U$ in $\GL_r(k)$, and $\Sol$ be the solutions of the equation 
$\overline{U}\cdot x^{(p)}=x.$ As $k$ is algebraically closed, $\Sol$ is an $\F_p$-space of dimension $r$, and $\Lie(G^{\et\vee})\otimes k$ is generated by $\Sol$ (cf. \cite[Prop. 4.1]{Ka}).  By the henselian property of $A$, every elements in $\Sol$ lifts uniquely to a solution of $U\cdot x^{(p)}=x$, \ie the reduction  map $\Lie(G^\vee)^{\varphi=1}\xra{\sim} \Sol$ is bijective. By Nakayama's lemma, $\Lie(G^\vee)^{\varphi=1}$ generates the $A$-module $\Lie(G^{\et\vee})$. 
\end{proof}

\subsection{} We keep the notations of \ref{nota-one-dim}. Let  $\Comp$ be the category of neotherian complete local $\Kb_0$-algebras with residue field $\Kb_0$,  $\D_{\cG_{\Kb_0}}$ (\resp $\D_{\cG^{\circ}_{\Kb_0}}$) be the  functor which associates to every object $A$ of $\Comp$ the set of isomorphic classes of deformations of $\cG_{\Kb_0}$ (\resp $\cG^\circ_{\Kb_0}$) . If $A$ is an object in $\Comp$ and $G$ is a deformation  of $\cG_{\Kb_0}$ (\resp $\cG^\circ_{\Kb_0}$) over $A$, we denote by $[G]$ its isomorphic class in $\D_{\cG_{\Kb_0}}(A)$ (\resp in $\D_{\cG^\circ_{\Kb_0}}$).

\begin{lemma}\label{lemma-one-dim-2}  Let $\Sigma$ be the set defined in \eqref{defn-Sigma}.

\emph{(i)} The morphism of sets $\Phi:\Sigma\ra \D_{\cG_{\Kb_0}}(\Rb)$ given by $\sigma\mapsto [\cG_{\Rb,\sigma}]$ is bijective.

\emph{(ii)}  Let $\sigma \in \Sigma$. Then there exists a basis of $\Lie(\cG^{\circ\vee}_{\Rb,\sigma})$ such that $\varphi_{\cG^{\circ}_{\Rb,\sigma}}$ is represented by a matrix of the form
\begin{equation}\label{congru-matrix}
\h^\circ_{\sigma}= \begin{pmatrix}0& 0 &\cdots&0 &a_1\\
 1&0&\cdots &0 &a_2\\
 \vdots&&\ddots&& \vdots\\
0&0&\cdots &1& a_{n-1}\end{pmatrix}
\end{equation}
with $a_i\equiv \alpha\cdot \sigma(t_i)\;(\mathrm{mod}\, \m^2_{\Rb})$ for $1\leq i \leq n-1$, where $\alpha\in \Rb^\times$ and $\m_{\Rb}$ is the maximal ideal of $\Rb$. In particular,
 $\cG^\circ_{\Rb,\sigma}$ is the universal deformation of $\cG^\circ_{\Kb_0}$ if and only if $\{\sigma(t_1),\cdots,\sigma(t_{n-1})\}$ is a system of regular parameters of $\Rb$.
\end{lemma}

\begin{proof} (i) We begin with a remark on the Kodaira-Spencer map of $\cG_{R'}$. Let $\cT_{\bS/k}=\cHom_{\cO_{\bS}}(\Omega^1_{\bS/k},\cO_{\bS})$ be the tangent sheaf of $\bS$.  Since $\bG$ is universal, the Kodaira-Spencer map \eqref{Kod-map}
\[\Kod: \cT_{\bS/k}\xra{\sim} \cHom_{\cO_{\bS}}(\omega_{\bG},\Lie(\bG^\vee)) \]
 is an isomorphism.  By functoriality, this induces an isomorphism of $R'$-modules
\begin{equation}\label{Kod-R'}
\Kod_{R'}:T_{R'/k}\xra{\sim} \Hom_{R'}(\omega_{\cG_{R'}},\Lie(\cG_{R'}^\vee)), 
\end{equation}
where $T_{R'/k}=\Hom_{R'}(\Omega^1_{R'/k},R')=\Gamma(\bS,\cT_{\bS/k})\otimes_R R'$.

For each integer $\nu\geq 0$, we put $\Rb_{\nu}=\Rb/\m_{\Rb}^{\nu+1}$, $\Sigma_{\nu}$ to be the set of liftings of $R\ra K_0\ra \Kb_0$ to $R\ra \Rb_{\nu}$, and  $\Phi_{\nu}:\Sigma_{\nu}\ra \D_{\cG_{\Kb_0}}(\Rb_\nu)$ to be the morphism of sets $\sigma_{\nu}\mapsto [\cG_{R'}\otimes_{\sigma_{\nu}} \Rb_{\nu}]$. We prove by induction on $\nu$ that $\Phi_\nu$ is bijective  for all $\nu\geq 0$. This will complete the proof of (i). For $\nu=0$, the claim holds trivially. Assume that it holds for $\nu-1$ with $\nu\geq 1$. We have a commutative diagram
\[\xymatrix{\Sigma_{\nu}\ar[d]\ar[rr]^{\Phi_{\nu}}&&\D_{\cG_{\Kb_0}}(\Rb_{\nu})\ar[d]\\
\Sigma_{\nu-1}\ar[rr]^{\Phi_{\nu-1}}&&\D_{\cG_{\Kb_0}}(\Rb_{\nu-1}),}\]
where the vertical arrows are the canonical reductions, and the lower arrow is an isomorphism by induction hypothesis. Let $\tau$ be an arbitrary element of $\Sigma_{\nu-1}$. We denote by $\Sigma_{\nu,\tau}\subset \Sigma_{\nu}$ the preimage of  $\tau$, and by $\D_{\Phi_{\nu-1}(\tau)}(\Rb_\nu)\subset \D_{\cG_{\Kb_0}}(\Rb_\nu)$ the preimage of $\Phi_{\nu-1}(\tau)$. It suffices to prove that $\Phi_{\nu}$ induces a bijection between $\Sigma_{\nu,\tau}$ and $\D_{\Phi_{\nu-1}(\tau)}(\Rb_\nu)$.  Let $I_{\nu}=\m_{\Rb}^{\nu}/\m_{\Rb}^{\nu+1}$ be the ideal of  the reduction map  $\Rb_{\nu}\ra \Rb_{\nu-1}$. By [EGA $\mathrm{0_{IV}}$ 21.2.5 and 21.9.4], we have  $\Omega^1_{R'/k}\simeq \widehat{\Omega}^1_{R'/k}$, and they are free over $A$ of rank $n$.  By [EGA $\mathrm{0_{IV}}$ 20.1.3], $\Sigma_{\nu,\tau}$ is a (nonempty) homogenous space under the group 
$$
\Hom_{K_0}(\Omega^1_{R'/k}\otimes_{R'}K_0, I_{\nu})=T_{R'/k}\otimes_{R'}I_{\nu}.
$$ 
On the other hand,  according to \ref{prop-deform}(i), $\D_{\Phi_{\nu-1}(\tau)}(\Rb_{\nu})$ is a homogenous space under the group 
$$\Hom_{\Kb_0}(\omega_{\cG_{\Kb_0}},\Lie(\cG^{\vee}_{\Kb_0}))\otimes_{\Kb_0}I_{\nu}=\Hom_{R'}(\omega_{\cG_{R'}},\Lie(\cG^\vee_{R'}))\otimes_{R'} I_{\nu}.$$
 Moreover, it is easy to check that the morphism of sets $\Phi_{\nu}:\Sigma_{\nu,\tau}\ra \D_{\Phi_{\nu-1}(\tau)}(\Rb_{\nu})$ is compatible with the homomorphism of groups
\[\Kod_{R'}\otimes_{R'}\Id:T_{R'/k}\otimes_{R'}I_{\nu}\ra \Hom_{R'}(\omega_{\cG_{R'}},\Lie(\cG^\vee_{R'}))\otimes_{R'} I_{\nu},\]
where $\Kod_{R'}$ is  the Kodaira-Spencer map \eqref{Kod-R'} associated to $\cG_{R'}$. The bijectivity of $\Phi_{\nu}$ now follows from the fact that $\Kod_{R'}$ is an isomorphism.

 (ii) First, we determine the submodule $\Lie(\cG^{\et\vee}_{\Rb,\sigma})$ of $\Lie(\cG^\vee_{\Rb,\sigma})$. We choose a basis of $\Lie(\bG^\vee)$ over $\cO_{\bS}$ such that $\varphi_{\bG}$ is expressed by the matrix $\h$ \eqref{HW-one-dim}. As $\cG_{\Rb,\sigma}$ derives from $\bG$ by base change $R\ra R'\xra{\sigma}\Rb$,  there exists a basis $(e_1,\cdots,e_n)$ of $\Lie(\cG_{\Rb,\sigma}^\vee)$ such that $\varphi_{\cG_{\Rb,\sigma}}$ is expressed by 
\[\h^{\sigma}=\begin{pmatrix}
0&0 &\cdots& 0& -\sigma(t_1)\\
1&0&\cdots& 0& -\sigma(t_2)\\
\vdots& &\ddots &&\vdots\\
0 &0&\cdots &1&-\sigma(t_n) 
\end{pmatrix}.
\] By Lemma \ref{sublemma-1}, $\Lie(\cG^{\et\vee}_{\Rb,\sigma})$ is generated by $\Lie(\cG^{\vee}_{\Rb,\sigma})^{\varphi=1}$. If $\sum_{i=1}^nx_ne_n\in \Lie(\cG_{\Rb,\sigma}^\vee)^{\varphi=1}$ with $x_i\in \Rb$ for $1\leq i\leq n$, then $(x_i)_{1\leq i\leq n}$ must satisfy the equation
$\h^\sigma\cdot 
\begin{pmatrix}x^p_1\\
\vdots \\
x^p_n\end{pmatrix}=\begin{pmatrix}x_1\\ \vdots \\ x_n
\end{pmatrix};$
or equivalently,
\begin{equation}\label{equ-sol}
\begin{cases}x_1=-\sigma(t_1)x_n^p\\
x_2=-\sigma(t_2)x_n^p-\sigma(t_1)^px_n^{p^2}\\
\cdots\\
x_{n-1}=-\sigma(t_{n-1})x_n^p-\cdots-\sigma(t_1)^{p^{n-2}}x_n^{p^{n-1}}\\
\sigma(t_1)^{p^{n-1}}x_n^{p^n}+\sigma(t_2)^{p^{n-2}}x_n^{p^{n-1}}+\cdots +\sigma(t_n)x_n^{p}+x_n=0.
\end{cases}
\end{equation}
We note that $\sigma(t_i)\in \m_{\Rb}$ for $1\leq i\leq n-1$ and $\sigma(t_n)\in\Rb^\times$ with image $i(t_n)\in \Kb_0$, where $i:K_0\ra \Kb_0$ is the fixed immbedding. By Hensel's lemma, every solution in $\Kb_0$ of the equation $i(t_n)x_n^{p}+x_n=0$ lifts uniquely to a solution of \eqref{equ-sol}. As $\Lie(\cG^{\et\vee}_{\Rb,\sigma})$ has rank $1$, by Lemma \ref{sublemma-1}, these are all the solutions. Let $(\lambda_1,\cdots,\lambda_n)$ be a non-zero solution of \eqref{equ-sol}. We have 
\begin{equation}\label{congru-lam}\lambda_n\in \Rb^\times\quad \text{and}\quad \lambda_i\equiv -\lambda_n^p\sigma(t_i)\quad(\mathrm{mod}\;\m_{\Rb}^2).\end{equation}
We put $v=\lambda_1e_1+\cdots+\lambda_ne_n$;  so $v$ is a basis of $\Lie(\cG^{\et\vee}_{\Rb,\sigma})$ by \ref{sublemma-1}.  For $1\leq i\leq n$, let $f_i$ be the image of $e_i$ in $\Lie(\cG^{\circ\vee}_{\Rb,\sigma})$. Then $f_1,\cdots,f_n$ clearly generate $\Lie(\cG^{\circ\vee}_{\Rb,\sigma})$. By the explicit description above of $\Lie(\cG^{\et\vee}_{\Rb,\sigma})$, we have $f_n=-\lambda^{-1}_n(\lambda_1f_1\cdots+\lambda_{n-1}f_{n-1})$. Hence $f_1,\cdots,f_{n-1}$ form  a basis of $\Lie(\cG^{\circ\vee}_{\Rb,\sigma})$. By the functoriality of Hasse-Witt maps, we have $\varphi_{\cG^{\circ}_{\Rb}}(f_i)=f_{i+1}$ for $1\leq i\leq n-1$, or equivalently,
\[\varphi_{\cG^\circ_{\Rb,\sigma}}(f_1,\cdots,f_{n-1})=(f_1,\cdots,f_{n-1})\cdot 
\begin{pmatrix}0&0&\cdots&0&-\lambda_n^{-1}\lambda_1\\
1&0&\cdots&0&-\lambda_n^{-1}\lambda_2\\
\vdots &&\ddots&&\vdots\\
0&0&\cdots&1&-\lambda_n^{-1}\lambda_{n-1}
\end{pmatrix}.\]
In view of \eqref{congru-lam}, we see that the above matrix has the form of \eqref{congru-matrix} by setting $\alpha=\lambda_n^{p-1}\in \Rb^\times$. The second part of statement (ii) follows immediately from Proposition \ref{prop-HW-versal}(ii) and the description above of $\varphi_{\cG^{\circ}_{\Rb,\sigma}}$.
\end{proof}

\begin{lemma}\label{lemma-lifting}
Let $F$ be a field with the discrete topology, $A$ be a noetherian local complete and formally smooth $F$-algebra, $C$ be an adic topological $F$-algebra, $J\subset C$ be an ideal of definition $($\ie $C=\varprojlim_{n}C/J^{n+1})$, $g:A\ra C/J$ be a continuous homomorphism of topological $F$-algebras. Let $t_1,\cdots,t_n$ be elements in $A$ such that $dt_1,\cdots, dt_{n}$ form a basis of $\widehat{\Omega}^1_{A/F}$ over $A$, and $a_1,\cdots, a_n\in C$ be such that the image of $a_i$ in $C/J$ is $g(t_i)$ for $1\leq i\leq n$. Then there exists a unique continuous homomorphism of topological $F$-algebras $h:A\ra C$ which lifts  $g$ and satisfies $h(t_i)=a_i$ for $1\leq i\leq n$.
\end{lemma}

\begin{proof} For each integer $\nu\geq 0$, we put $C_{\nu}=C/J^{\nu+1}$. It suffices to prove that there exists, for every integer $\nu\geq 0$, a unique continuous homomorphism of topological $F$-algebras $h_\nu:A\ra C_\nu$ which lifts $g=h_0$ and verifies $h_\nu(t_i)\equiv a_i\quad (\mathrm{mod}\;J^{\nu+1})$. We proceed by induction on $\nu\geq 0$. For $\nu= 0$, the assertion is trivial. Suppose that $\nu\geq 1$ and the  required homomorphism $h_{\nu-1}:A\ra C_{\nu-1}$ exists uniquely. Since $A$ is formally smooth over $F$, by [EGA $\mathrm{0_{IV}}$ 20.7.14.4 and 20.1.3], the set of continuous homomorphisms $A\ra C_{\nu}$ lifting $h_{\nu-1}$ is a homogeneous space under the group 
$\mathrm{Hom.cont}_{A}(\widehat{\Omega}^1_{A/F},J^{\nu}/J^{\nu+1}),$
where $\mathrm{Hom.cont}_A$ denotes the group of continuous homomorphisms of topological modules over $A$. Since $C/J$ is a discrete topological ring, there exists an inteter $\ell\geq 0$, such that the continuous map $g:A\ra C/J$ factors through the canonical surjection $A\ra A/\m_A^{\ell}$, where $\m_A$ is the maximal ideal of $A$. Note that $J^{\nu}/J^{\nu+1}$ is a $C/J$-module; so we have 
\[\mathrm{Hom.cont}_A(\widehat{\Omega}^1_{A/F},J^{\nu}/J^{\nu+1})=\Hom_{A/\m_A^{\ell}}(\widehat{\Omega}^1_{A/F}\otimes A/\m_{A}^\ell,J^{\nu}/J^{\nu+1}).\] 
Now let $\widetilde{h}_{\nu}:A\ra C_{\nu}$ be an arbitrary continuous lifting of $h_{\nu-1}$; then any other liftings of $h_{\nu-1}$ to $C_{\nu}$ writes as $\widetilde{h}_{\nu}+\delta$ with $\delta \in \Hom_{A/\m_A^{\ell}}(\widehat{\Omega}^1_{A/F}\otimes A/\m_{A}^\ell,J^{\nu}/J^{\nu+1})$. By assumption, $dt_1,\cdots, dt_{n}$ being a basis of $\widehat{\Omega}^1_{A/F}$, there exists thus a unique $\delta_0$ such that 
$\delta_0(t_i)\equiv a_i-\widetilde{h}_{\nu}(t_i) \quad (\mathrm{mod}\; J^{\nu+1}).$
Then $h_{\nu}=\widetilde{h}_{\nu}+\delta_0$ is the unique continuous homomorphism $A\ra C_{\nu}$ lifting $g$ and satisfying $h_\nu(t_i)\equiv a_i \quad (\mathrm{mod}\; J^{\nu+1}).$ This completes the induction.
\end{proof}

Now we can turn to the proof of \ref{lemma-key}.

\subsection{Proof of Lemma \ref{lemma-key}} First,  suppose that we have found a $\sigma_2\in \Sigma$ such that $\cb_{\sigma_2}\neq 0$ and $\cG^\circ_{\Rb,\sigma_2}$ is the universal deformation of $\cG^\circ_{\Kb_0}$.  Since $\Phi:\Sigma\xra{\sim} \D_{\cG_{\Kb_0}}(\Rb)$ is bijective by \ref{lemma-one-dim-2}(i), there  exists a $\sigma_1\in \Sigma$ corresponding to the deformation $[\cG^\circ_{\Rb,\sigma_2}\oplus \Q_p/\Z_p]\in \D_{\cG_{\Kb_0}}(\Rb)$. It is clear that $\cG^{\circ}_{\Rb,\sigma_1}\simeq \cG^{\circ}_{\Rb,\sigma_2}$. Besides,  the exact sequence \eqref{filt-Tate-R'} for  $\sigma_1$ splits; so we have $C_{\sigma_1}=0$. It remains to prove the existence of $\sigma_2$. We note first that $\Kb_0$ can be canonically imbedded  into $\Rb$, since it is perfect. Since $R'$ is formally smooth over $k$ and $(dt_1,\cdots,dt_n)$ is a basis of $\widehat{\Omega}^1_{R'/k}\simeq \Omega^1_{R'/k}$, Lemma \ref{lemma-lifting} implies that there is a $\sigma \in \Sigma$ such that $\sigma(t_i)\;(1\leq i\leq n-1)$ form a system of regular parameters of $\Rb$ and $\sigma(t_n)\in \Kb_0\subset \Rb$. We claim that  $\sigma_2=\sigma$  answers the question. In fact,  Lemma \ref{lemma-one-dim-2}(ii) implies that $\cG^\circ_{\Rb,\sigma}$ is the universal deformation of $\cG^\circ_{\Kb_0}$. It remains to verify  that $\cb_{\sigma}\neq 0$.

 Let $A=\Kb_0[[\pi]]$ be a complete discrete valuation ring of characteristic $p$ with residue field $\Kb_0$, $T=\Spec(A)$, $\xi$ be the generic point of $T$, $\xib$ be a geometric over $\xi$, and $I=\Gal(\xib/\xi)$ the Galois group. We define  a homomorphism of $\Kb_0$-algebras $f^*:\Rb\ra A$ by putting $f^*(\sigma(t_1))=\pi$ and $f^*(\sigma(t_i))=0$ for $2\leq i\leq n-1$. This is possible, since $(\sigma(t_1),\cdots, \sigma(t_{n-1}))$ is a system of regular parameters of $\Rb$.   Let $f:T\ra \Sb$ be the homomorphism of schemes corresponding to $f^*$, and $\cG_T=\cG_{\Rb,\sigma}\times_{\Sb}T$. By the functoriality of Hasse-Witt maps, 
\[\h_{T}=\begin{pmatrix}0&0&\cdots&0 &-\pi\\
1&0&\cdots&0&0\\
0&1&\cdots&0&0\\
\vdots&&\ddots&&\vdots\\
0&0&\cdots&1&-f^*(\sigma(t_n))\end{pmatrix}\in \rM_{n\times n}(\Rb)\]
is a matrix of $\varphi_{\cG_{T}}.$ By definition \eqref{defn-hw-index}, the Hasse invariant of $\cG_T$ is $h(\cG_{T})=1$. In particular, $\cG_T$ is generically ordinary. Let $\widetilde{U}'_\sigma\subset \Sb$ be the ordinary locus of $\cG_{\Rb,\sigma}$. We have $f(\xi)\in \widetilde{U}'_\sigma$. By the functoriality of fundamental groups, $f$ induces a homomorphism of groups
\[\pi_1(f):I=\Gal(\xib/\xi)\ra \pi_1(\widetilde{U}'_\sigma,f(\xib))\simeq \pi_1(\widetilde{U}'_\sigma,\xb).\]
 Let $\cG^\circ_T$ be the connected part of $\cG_T$, and $\cG_{T}^\et$ be the \'etale part of $\cG_{T}$. Then $\cG^\et_{T}\simeq\Q_p/\Z_p$. We have an exact sequence of $\F_p[I]$-modules
\[0\ra \cG^\circ_T(1)(\xib)\ra \cG_T(1)(\xib)\ra \cG^\et_T(1)(\xib)\ra 0,\]
which determines a cohomology class $\cb_T\in H^1(I,\cG_T^\circ(1)(\xib))$. We notice that $\cG_T(1)(\xib)$ is isomorphic to $\cG_{\Rb,\sigma}(1)(\xb)$ as an abelian group, and the action of $I$ on $\cG_T(1)(\xib)$ is induced by the action of $\pi_1(\widetilde{U}'_\sigma,\xb)$ on $\cG_{\Rb,\sigma}(1)(\xb)$. Therefore, $\cb_T$ is the image of $\cb_\sigma$ by the functorial map
\[H^1\bigl(\pi_1(\widetilde{U}'_\sigma,\xb),\cG^\circ_{\Rb,\sigma}(1)(\xb)\bigr)\ra H^1\bigl(I,\cG^\circ_T(1)(\xib)\bigr).\]
To verify that $\cb_\sigma\neq 0$, it suffices to check that $\cb_T\neq0$.  We consider the polynomial $P(X)=X^{p^n}+f^*(\sigma(t_n))X^{p^{n-1}}+\pi X\in A[X]$. According to \ref{cor-non-zero-coh}, it suffices to find a $\alpha\in \Kb_0\subset A$ such that $P(\alpha)$ is a uniformizer of $A$. But by the choice of $\sigma$, we have $\sigma(t_n)\in \Kb_0$ and $\sigma(t_n)\neq 0$; so $f^*(\sigma(t_n))\neq 0$ lies in $\Kb_0$. Let  $\alpha$ be a $p^{n-1}(p-1)$-th root of $-f^*(\sigma(t_n))$ in $\Kb_0$. Then  we have $\alpha\in \Kb_0^\times$, and $P(\alpha)=\alpha\pi$ is a uniformizer of $A$. This completes the proof of  \ref{lemma-key}.

\section{End of the Proof of Theorem \ref{thm-main}}

In this section, $k$ denotes an algebraically closed field of characteristic $p>0$.
\subsection{}\label{New-strata} First, we recall some preliminaries on  Newton stratification due to F. Oort. Let $G$ be an arbitrary BT-group over $k$, $\bS$ be the local moduli of $G$ in characteristic $p$, and $\bG$ be the universal deformation of $G$ over $\bS$ \eqref{defn-moduli}. Put $d=\dim (G)$ and $c=\dim(G^\vee)$. We denote by $\cN(G)$  the Newton polygon of $G$ which has endpoints $(0,0)$ and $(c+d,d)$. Here we use the normalization of Newton polygons such that slope 0 corresponds to \'etale BT- groups and slope 1 corresponds to groups of multiplicative type.

Let $\Nt(c+d,d)$ be the set of Newton polygons with endpoints $(0,0)$ and $(c+d,d)$ and slopes in $(0,1)$.  For $\alpha,\beta \in \Nt(c+d,d)$, we say that $\alpha\preceq \beta$ if no point of $\alpha$ lies below $\beta$; then ``$\preceq$'' is a partial order on $\Nt(c+d,d)$. For each $\beta\in \Nt(c+d,d)$, we denote by $V_\beta$ the subset of $\bS$ consisting of points $x$ with $\cN(\bG_x)\preceq\beta$, and by $V_\beta^\circ$ the subset of $\bS$ consisting of points $x$ with $\cN(\bG_x)=\beta$. By Grothendieck-Katz's specialization theorem of Newton polygons, $V_\beta$ is closed in $\bS$, and $V_\beta^\circ$ is open (maybe empty) in $V_\beta$. We put 
$$
\diamondsuit(\beta)=\{(x,y)\in \Z\times \Z\;|\; 0\leq y<d, y<x<c+d, (x,y) \text{ lies on or above the polygon } \beta\},
$$ and $\dim(\beta)=\#(\diamondsuit(\beta))$.
 
\begin{thm}[\cite{oort2} Theorem 2.11]\label{thm-oort} Under the above assumptions, for each $\beta\in \Nt(c+d,d)$, the subset $V^\circ_\beta$ is non-empty if and only if $\cN(G)\preceq \beta$. In that case, $V_\beta$ is the closure of $V^\circ_\beta$ and 
all irreducible components of $V_\beta$ have dimension $\dim(\beta)$.
\end{thm} 
 
\subsection{} Let $G$ be a connected
and HW-cyclic BT-group over $k$ of dimension $d=\dim(G)\geq 2$. Let  $\beta\in \Nt(c+d,d)$ be the Newton polygon given by the following slope sequence:
\[ \beta=(\underbrace{1/(c+1),\cdots,1/(c+1)}_{c+1},\underbrace{1,\cdots,1}_{d-1}).\]
We have $\cN(G)\preceq \beta$ since $G$ is supposed to be
connected. By Oort's Theorem \ref{thm-oort}, $V_\beta$ is a equal dimensional
closed subset of the local moduli $\bS$ of  dimension $c(d-1)$. We endow $V_\beta$ with the structure of  a reduced closed subscheme of $\bS$. 

\begin{lemma} Under the above assumptions, let $R$ be the ring of $\bS$,
 and 
\[
\begin{pmatrix}0 &0 &\cdots &0 &-a_1\\
1 &0 &\cdots &0 &-a_2\\
0&1&\cdots&0&-a_3\\
\vdots &&\ddots &&\vdots\\
0&0&\cdots&1&-a_c\end{pmatrix}\in \rM_{c\times c}(R)
\]
  be a matrix of the Hasse-Witt map $\HW_G$. Then the closed reduced subscheme $V_\beta$  of $\bS$ is defined by the prime ideal  $(a_1,\cdots,a_c)$. In particular, $V_\beta$ is irreducible.
\end{lemma} 

\begin{proof} Note first that $\{a_1,\cdots, a_c\}$ is a subset of a system of regular parameters of $R$ by 
\ref{prop-HW-versal}(i). Let  $I$ be the ideal of $R$ defining $V_\beta$.  Let  $x$ be an arbitrary point of $V_\beta$, we denote by $\fp_x$ the prime ideal of $R$ corresponding to $x$.  Since the Newton polygon of the fibre $\bG_x$ lies above
$\beta$, $\bG_x$ is connected. By Lemma \ref{lemma-HW-V}, we have $a_i\in \fp_x$
for $1\leq i\leq c$. Since $V_\beta$ is reduced, we have $a_i\in
I$. Let $\mathfrak{P}=(a_1,\cdots, a_c)$, and $V(\mathfrak{P})$ the closed subscheme of
$\bS$ defined by $\mathfrak{P}$. Then $V(\mathfrak{P})$ is an integral scheme of
dimension $c(d-1)$ and $V_\beta\subset V(\mathfrak{P})$. Since Theorem
\ref{thm-oort} implies that $\dim V_\beta=c(d-1)$,  we have necessarily $V_\beta=V(\mathfrak{P})$. 
\end{proof}

We keep the assumptions above. Let $(t_{i,j})_{1\leq i\leq c,
  1\leq j \leq d}$ be  a regular system of
parameters of $R$ such
that $t_{i,d}=a_i$ for all $1\leq i \leq c$. Let $x$ be the generic
point of the Newton strata $V_\beta$,  $k'=\kappa(x)$, and
$R'=\widehat{\cO}_{\bS,x}$. Since $R$ is noetherian and integral, the
canonical ring homomorphism $R\ra \cO_{\bS,x}\ra R'$ is injective.
The image in $R'$ of an element $a\in
R$ will be denoted also by $a$. By choosing a $k$-section $k'\ra R'$ of the canonical projection
$R'\ra k'$, we get a (non-canonical) isomorphism of $k$-algebras $R'\simeq
k'[[t_{1,d},\cdots, t_{c,d}]]$. Let $k''$ be an algebraic closure of
$k'$, and $R''=k''[[t_{1,d},\cdots,t_{c,d}]]$. Then we have a natural
injective homomorphism of $k$-algebras $R'\ra R''$ mapping $t_{i,d}$
 to $t_{i,d}$ for  $ 1\leq i\leq c$. 

Let $S''=\Spec (R'')$, $\xb$ be its closed point. By the construction of $S''$, we have a morphism
of $k$-schemes 
\begin{equation}f: S''\ra \bS \end{equation}
sending $\xb$ to $x$. We put $\cG=\bG\times_{\bS} S''$. By the choice of the Newton polygon $\beta$, the closed fibre $\cG_{\xb}$ has a  BT-subgroup $\cH_{\xb}$ of multiplicative type of height $d-1$. Since $S''$ is henselian, $\cH_{\xb}$ lifts uniquely to a BT-subgroup $\cH$  of $\cG$. We put $\cG''=\cG/\cH$. It is a connected BT-group over $S''$ of dimension $1$ and height $c+1$. \\

\begin{lemma}\label{rem-Lau} Under the above assumptions, $\cG''$ is the universal deformation in equal characteristic of its special fiber. 
\end{lemma}  
 This lemma is a particular case of \cite[Lemma 3.1]{lau}. Here, we use \ref{prop-HW-versal}(ii) to give a simpler proof.

\begin{proof} We have an exact sequence of BT-groups over $S''$
\[0\ra \cH\ra \cG\ra \cG''\ra 0,\]
which induces an exact sequence of Lie algebras $0\ra
\Lie(\cG''^\vee)\ra \Lie(\cG^\vee)\ra \Lie(\cH^\vee)\ra 0$ compatible
with Hasse-Witt maps. Since $\cH$ is of multiplicative type, we get
$\Lie(\cH^\vee)=0$ and  an isomorphism of Lie algebras
$\Lie(\cG''^\vee)\simeq \Lie(\cG^\vee)$.  By the choice of the regular system $(t_{i,j})_{1\leq i \leq c, 1\leq j \leq d}$, there is  a basis $(v_1,\cdots, v_c)$ of $\Lie(\cG''^\vee)$ over $\cO_{S''}$ such that $\HW_{\cG''}$ is given by the matrix 
\[\h=\begin{pmatrix}
0 &0 &\cdots &0 &-t_{1,d}\\
1 &0 &\cdots &0 &-t_{2,d}\\
0 &1 &\cdots &0 &-t_{3,d}\\
\vdots &&\ddots &&\vdots\\
0&0&\cdots&1&-t_{c,d}\end{pmatrix}.\]
Now the lemma results from Proposition \ref{prop-HW-versal}(ii).
\end{proof}

\subsection{Proof of Theorem \ref{thm-main}} The one-dimensional case is treated in \ref{thm-one-dim}. If
  $\dim(G)\geq 2$, we apply the preceding discussion to obtain the
morphism $f\colon S''\ra \bS$ and the BT-groups $\cG=\bG\times_{\bS}S''$ and $\cG''$,
which is the quotient  of $\cG$ by the maximal subgroup of $\cG$ of
multiplicative type.   Let $U''$ be the common ordinary locus of $\cG$
and $\cG''$ over
$S''$, and $\xib$ be a geometric point of $U''$. Then $f$ maps $U''$ into the ordinary locus $\bU$ of $\bG$. We denote by 
 $$
\rho_\cG:\pi_1(U'',\xib)\ra \Aut_{\Z_p}(\rT_p(\cG,\xib))
$$ the monodromy representation associated to $\cG$, and the same notation for $\rho_{\cG''}$. By   the functoriality of monodromy, we have $\im(\rho_{\cG})\subset\im(\rho_{\bG})$. On the other hand, the canonical map $\cG\ra \cG''$ induces an isomorphism of Tate modules
$\rT_p(\cG,\etab)\xra{\sim}\rT_p(\cG'',\etab)$ compatible with the action of $\pi_1(U'',\etab)$. Therefore, the group $\im(\rho_{\cG})$ is identified with $\im(\rho_{\cG''})$. Since $\cG''$ is one-dimensional,  we conclude the proof by Lemma \ref{rem-Lau} and Theorem \ref{thm-one-dim}.


\begin{thebibliography}{99}

\bibitem{AN} \textsc{J. Achter} and \textsc{P. Norman}, Local
monodromy of $p$-divisible groups, Preprint, (2006).


\bibitem{BBM} \textsc{P. Berthelot, L. Breen} and \textsc{W.
Messing},  \emph{Th\'eorie de Dieudonn\'e Cristalline II}, Lect.
notes in Math. \textbf{930}, Springer-Verlag, (1982).

\bibitem{Bou} \textsc{N. Bourbaki}, \emph{Alg\`ebre Commutative},
Masson, Paris (1985).

\bibitem{Ch2} \textsc{L. Chai}, Local monodromy for deformations of one dimensional formal groups, \emph{J. reine angew. Math.} \textbf{524}, (2000), 227-238.

\bibitem{Ch} \textsc{L. Chai}, Methods for $p$-adic monodromy,
to appear in \emph{Jussieu J. Math.}, (2006).

\bibitem{CO} \textsc{L. Chai} and \textsc{F. Oort}, Monodromy and irreducibility of leaves, available at the webpage of F. Oort, (2008).

\bibitem{DR} \textsc{P. Deligne} and \textsc{K. Ribet}, Values of
abelian L-functions at negative integers over totally real fields.
\emph{Inven. Math.} \textbf{59}, (1980), 227-286.

\bibitem{De} \textsc{M. Demazure}, \emph{Lectures on $p$-Divisible
Groups}, Lect. notes in Math. \textbf{302}, Springer-Verlag, (1972).

\bibitem{DG} \textsc{M. Demazure} and \textsc{A. Grothendieck},
\emph{Sch\'ema en Groupes} I (SGA $\mathrm{3_I}$) , Lect. notes in
Math. \textbf{151}, Springer-Verlag, (1970).

\bibitem{Ek} \textsc{T. Ekedahl}, The action of monodromy on torsion
points of Jacobians, \emph{Arithmetic Algebraic Geometry}, G. van
der Geer, F. Oort and J. Steenbrink, ed. Progress in Math.
\textbf{89}, Birkh\"auser, (1991), 41-49.

\bibitem{FC} \textsc{G. Faltings} and \textsc{L. Chai}, \emph{Degeneration of Abelian
Varieties}, Ergebnisse Bd \textbf{22}, Springer-Verlag,(1990).

\bibitem{Go} \textsc{H. Gross}, Ramification in $p$-adic Lie
extensions, \emph{Journ\'ee de G\'eom\'etrie Alg\'ebrique de Rennes
III}, \emph{Ast\'erisque} \textbf{65}, (1979), 81-102.

\bibitem{Gr} \textsc{A. Grothendieck}, \emph{Groupes de Barsotti-Tate et Cristaux de
Dieudonn\'e}, les Presses de l'Universit\'e de Montr\'eal, (1974).

\bibitem{Hi} \textsc{H. Hida} $p$-adic automorphic forms on
reductive groups, \emph{Ast\'erisque} \textbf{296} (2005), 147-254.

\bibitem{Il} \textsc{L. Illusie}, D\'eformations de groupes de
Barsotti-Tate (d'apr\`es A. Grothendieck), \emph{Ast\'erisque}
\textbf{127} (1985), 151-198.

\bibitem{Ig} \textsc{J. Igusa}, On the algebraic theory of elliptic
modular functions. \emph{J. Math. Soc. Japan} \textbf{20} (1968),
96-106.

\bibitem{dJ} \textsc{A. J. de Jong}, Crystalline Diedonn\'e module
theory via formal and rigide geometry, \emph{Publ. Math. Inst.
Hautes \'Etud. Sci.} \textbf{82} (1995), 5-96.

\bibitem{Kz} \textsc{N. Katz}, Algebraic solutions of differential equations ($p$-curvature and the Hodge filtration), \emph{Inven. Math.} \textbf{18} (1972), 1-118. 

\bibitem{Ka} \textsc{N. Katz}, $p$-adic properties of modular
schemes and modular forms, in \emph{Modular Functions of One
Variable III}, Lect. notes in Math. \textbf{350}, Springer-Verlag,
(1973).

\bibitem{lau} \textsc{E. Lau}, Tate modules and universal $p$-divisible groups, arXiv:0803.1390v1, (2008).


\bibitem{oort} \textsc{F.  Oort}, Newton polygons and formal groups: conjectures by Manin and Grothendieck, \emph{Ann. of Math.} \textbf{152} (2000), 183-206.

\bibitem{oort2}\textsc{F. Oort}, Newton polygon strata in the moduli space of abelian varieties, in \emph{Moduli space of Abelian Varieties} (Progress in Mathematics \textbf{195}), Birkh\"auser-Verlag, (2001), 417-440.

\bibitem{Se2} \textsc{J. P. Serre}, \emph{Abelian $\ell$-adic representations and elliptic curves}, A K Peters, Wellesley, MA, (1998). Originally published in 1968 by W. A. Benjamin.

\bibitem{Se} \textsc{J. P. Serre}, Propri\'et\'es galoisiennes des
points d'ordre fini des courbes elliptiques, \emph{Inven. Math.}
\textbf{15} (1972), 259-331.

\bibitem{CL} \textsc{J. P. Serre}, \emph{Corps Locaux}, Hermann,
Paris, (1968).

\bibitem{str} \textsc{M. Strauch}, Galois actions on torsion points of universal one-dimensional formal modules, arXiv: 0709.3542, (2007).

\bibitem{tian} \textsc{Y. Tian},  Thesis at University Paris 13, defensed in November 2007.

\bibitem{tian2} \textsc{Y. Tian},  $p$-adic monodromy of the universal deformation of an elementary Barsotti-Tate group, arXiv: 0708.2022, (2007).
\end{thebibliography}
\end{document}